\documentclass[12pt]{article}
\usepackage[right=2.5cm, left=2.5cm, top=3.3cm, bottom=3.3cm]{geometry}

\input{mathdefs.sty}
\usepackage{graphicx}
\usepackage{comment}

\begin{document}
\title{Ergodic Properties of $k$-Free\\ Integers in Number Fields}
\author{Francesco Cellarosi\footnote{University of Illinois at Urbana-Champaign. IL, U.S.A. \texttt{fcellaro@illinois.edu}},  Ilya Vinogradov\footnote{University of Bristol. U.K. \texttt{ilya.vinogradov@bristol.ac.uk}}}
\maketitle

\begin{abstract}
Let $K/\Q$ be a 
 degree $d$ extension. Inside the ring of integers $\roi$ we define the set of $k$-free integers $\mathcal F_k$ and 
a natural $\roi$-action on the space of binary $\roi$-indexed sequences, equipped with an $\roi$-invariant probability measure associated to $\mathcal F_k$.
We prove that this action is ergodic, has pure point spectrum, and is isomorphic to a $\Z^d$-action on a compact abelian group. In particular, it is not weakly mixing and has zero measure-theoretical entropy. This work generalizes the work of Cellarosi and Sinai [\emph{J. Eur. Math. Soc. (JEMS)} 15 (2013), no. 4, 1343--1374] that considered the case $K=\Q$ and $k=2$.
\end{abstract}


\section{Introduction}

It is an interesting question to study ``randomness'' of a given deterministic sequence. For a typical sequence coming from a chaotic dynamical system, such as doubling modulo one, one expects strong statistical properties, while for a circle rotation such properties cannot be expected. Of particular interest in this setting is the M\"obius sequence, $\{\mu(n)\}_{n\ge 1}$ defined as 
\be
\mu(n)=\begin{cases}1,&\mbox{if $n=1$};\\ 0,&\mbox{if $n$ is not square-free;}\\(-1)^m,&\mbox{if $n$ is the product of $m$ distinct primes.}\end{cases}
\nonumber
\ee
It is well known that 
\be\sum_{n\le N}\mu(n)=o(N),\label{sum-mu-PNT}\ee
suggesting that $\{\mu(n)\}_{n\geq 1}$ is reminiscent of a sequence of zero mean iid random variables, the above statement being the Law of Large Numbers for such a sequence. 

A related sequence, $\{\mu^2(n)\}_{n\geq 1}$, has been investigated by Sinai and the first author \cite{Cellarosi-Sinai-JEMS}. Their main result is that $\{\mu^2(n)\}_{n\geq 1}$, which is a sequence of zeros and ones, is generic for an ergodic subshift of infinite type on $\{0,1\}^\Z$ with pure point spectrum (see Section \ref{sec:purepoint} for more details). Such systems had been studied by von Neumann and Halmos \cite{vonNeumann-1932, Halmos-vonNeumann-1942}, and their statistical properties are well understood: they have zero measure-theoretical entropy and are not weakly mixing. In other words, the sequence $\{\mu^2(n)\}_{n\geq 1}$ has as little ``randomness'' as possible.

In the present paper we generalize the main result of \cite{Cellarosi-Sinai-JEMS} in two directions. Firstly, realizing that $\mu^2(n)$ is the indicator of square-free integers, we write $\mu^{(k)}(n)$ for the indicator of $k$-free numbers; that is, numbers that are not divisible by $p^k$ for every prime $p$. Secondly, we pass to a degree $d$ number field $K/\Q$ with ring of integers $\roi$ and define the M\"obius function on $\roi$. Since ideals $\fa$ in $\roi$ factor uniquely, we can define $\mu$  by 
\[\mu(\fa)=\begin{cases}1,&\mbox{if $\fa=\roi$};\\0,&\mbox{if $\fa$ is not square-free};\\(-1)^m,&\mbox{if $\fa$ is the product of $m$ distinct prime ideals}\end{cases}\]
and $\mu^{(k)}$
by 
\[\mu^{(k)}(\fa)=\begin{cases}1,&\mbox{$\fp^k\not\supseteq \fa$ for every prime ideal $\fp$};\\ 0,&\mbox{otherwise.}\end{cases}\]
Then, for $a\in\roi$, set $\mu(a)=\mu((a))$ and $\mu^{(k)}(a)=\mu^{(k)}((a))$. An ideal $\fa$ is \emph{$k$-free} if $\mu^{(k)}(\fa)=1$, while an integer $a\in\roi$ is \emph{$k$-free} if the principal ideal $(a)$ is $k$-free, and we denote the set of $k$-free integers in $\roi$ by $\mathcal F_k$. 
Thus $\{\mu^{(k)}(a)\}_{a\in\roi}$ is an $\roi$-indexed sequence of zeros and ones.

By an \emph{$\roi$-subshift} we mean a shift-invariant probability measure $\P$ on $X=\{0,1\}^{\roi}$ or, equivalently an action $\roi\curvearrowright(X,\P)$.  Let $\iota\colon(\Z^d,+)\to(\roi,+)$ be a group isomorphism, where $d$ is the degree of the extension $K/\Q$; it is unique up to multiplication by an element of $\mathrm{Aut}(\Z^d)$. The group $\Z^d$ acts via $\iota$ on the space of $\roi$-indexed sequences by $d$ commuting translations, and every $\roi$-subshift corresponds to a $\Z^d$-subshift. 
Let $B_x$ denote the ball of radius $x$ centered at the identity with respect to the $L^1$ norm induced on $\roi$ after identification with $\Z^d$ inside $\R^d$ via $\iota$. 
We say that a sequence $z=\{z(a)\}_{a\in\roi}\in\{0,1\}^{\roi}$ is \emph{generic} for an ergodic $\roi$-subshift $\P$ if the 
 ergodic theorem 
holds for $z$, i.e.\ for every $a_1,\ldots,a_r\in\roi$,
\beq\lim_{x\to\infty}\frac{1}{\#B_x}\sum_{a\in B_x} 
z(a+a_1)\cdots z(a+a_r)=\P\{w
\in\{0,1\}^{\roi}\colon 
w(a_1)=\ldots=w(a_r)=1\}\label{genericity}.
\eeq
In other words, genericity means that the frequency of every finite block equals the measure of the corresponding cylinder according to the subshift. This notion does not depend on the choice of $\iota$. 
A sequence $z$ for which the limit on the LHS of \eqref{genericity} exists is called \emph{stationary}. 
Given a stationary sequence $z$, the subshift $\P$ satisfying \eqref{genericity} is uniquely defined by Kolmogorov consistency \cite{Karatzas_Shreve, Koralov_Sinai} up to sets of measure zero, and in particular does not depend on $\iota$. 
Our main theorem states that the sequence $\{\mu^{(k)}(a)\}_{a\in\roi}$ is stationary and, more importantly, that the corresponding subshift is ergodic and has pure point spectrum. 

\begin{theorem}[Main Theorem, first version]\label{thm-1-intro}
Let $K/\Q$ be a degree $d$ extension. 
\begin{itemize}
 \item[(i)]
There exists a unique $\roi$-subshift $\Pi$ 
such that 
the sequence $\{\mu^{(k)}(a)\}_{a\in\roi}$ is 
generic for $\Pi$. This subshift is ergodic and has pure point spectrum. 

\item[(ii)] The $\roi$-subshift $\Pi$ is isomorphic to an action of $\Z^d$ by commuting translations on a compact abelian group equipped with the Haar measure. 
\end{itemize}
\end{theorem}

The proof of Theorem \ref{thm-1-intro} and of its full version Theorem \ref{main-thm-second-version} explicitly constructs the pure point subshift $\Pi$. The argument consists of three steps. 

First, we show the stationarity of the sequence $\{\mu^{(k)}(a)\}_{a\in\roi}$ by proving the existence of the asymptotic frequencies (correlation functions)
\beq c_{r+1}(a_1,\ldots,a_r)=\lim_{x\to\infty}\frac{1}{\#B_x}\sum_{a\in B_x}\mu^{(k)}(a)\mu^{(k)}(a+a_1)\cdots\mu^{(k)}(a+a_r),\label{eq:correlation}\eeq

We compute $c_{r+1}(a_1, \dots, a_r)$ explicitly (they do not depend on $\iota$) and give an error term for finite $x$ in Theorem \ref{th:M=S+O}.
This theorem is of independent interest, along with other explicit formul\ae\ given in Section \ref{sec3} (e.g.\ Proposition \ref{lem:Hall} generalizing a theorem by Hall \cite{Hall-1989}). 

A particular case  of \eqref{eq:correlation} is $c_2(0)=1/\zeta_K(k)$ (Corollary \ref{cor-density-k-free}), stating that the density of $k$-free integers in $\roi$ is $1/\zeta_K(k)$, where $\zeta_K$ is the Dedekind zeta function for the number field $K/\Q$. For $K=\Q$, the study of the average in \eqref{eq:correlation} as $x\to\infty$ is classical, see  \cite{Mirsky-1949, Hall-1989, Tsang-1986, Heath-Brown-1984}.

There is another notion of $k$-freeness for points in an arbitrary lattice studied by Baake, Moody, and Pleasants \cite{Baake-Moody-Pleasants-2000} and by Pleasants and Huck \cite{Pleasants-Huck}, for which the second correlation function, along with entropies and diffraction spectra, has been computed explicitly. This notion of $k$-freeness agrees with the one discussed above only when $K=\Q$.


The next step is to construct the compact abelian group 
\[\G=\prod_{\fp} \roi/\fp^2,\]
where the direct product ranges over prime ideals $\fp$ in $\roi$. We do this in Section \ref{Lambda-and-the-action-on-G} using only the \emph{second} correlation function, Bochner theorem, and Pontryagin duality. Since ideals thought of as additive subgroups have finite index, each factor $\roi/\fp^2$ is a finite group under addition. The Haar measure on $\G$ is simply the product of the counting measures on each factor. By identifying $\roi$ with $\Z^d$ as a group, and by choosing a basis for $\Z^d$, we get an action $\Z^d\curvearrowright(\G,\mathrm{Haar})$.
By construction, the spectrum of this action is pure-point, given by the countable group $\Lambda=\hat{\G}$ which can be identified with a subset of the $d$-dimensional torus.

In the third step (Section \ref{sec-spectrum-of-action}), we consider the unique probability measure $\Pi$ on $X=\{0,1\}^{\roi}$ whose finite dimensional marginals agree with the correlation functions above:
for every $r\geq0$ and every $a_0, a_1,\ldots,a_r\in\roi$
\beq
\label{eq:measurePi}
\Pi\left\{x\in X\colon  x(a_0)=x(a_1)=\ldots=x(a_r)=1\right\}=
c_{r+1}(a_1-a_0, a_2-a_0,\ldots, a_r-a_0), 
\eeq
up to normalization.
This defines a unique $\roi$-subshift (an action $\roi\curvearrowright(X,\Pi)$) for which the $d$-dimensional sequence $\{\mu^{(k)}(a)\}_{a\in\roi}$ is generic.
%
%
A substantial part of Section \ref{sec-spectrum-of-action} is dedicated to showing that the spectrum of the action $\roi\curvearrowright(X,\Pi)$ is given by $\Lambda$ (Theorem \ref{main-thm-second-version}). The method employed is constructive and uses explicit formul\ae\ for the two and three point correlation functions. 
Then, we apply a theorem of Mackey's \cite{Mackey-1964}, which states that two actions with pure point spectrum are isomorphic if and only if they are isospectral.
Since we know that $\Z^d\curvearrowright(\G,\mathrm{Haar})$ has spectrum $\Lambda$, Theorem \ref{thm-1-intro} follows from the isomorphism.
%
%

%
A consequence of the Main Theorem is the
\begin{corollary}\label{main-cor}
The subshift $\roi\curvearrowright(X,\Pi)$ in Theorem \ref{thm-1-intro} 
is not weakly mixing and it has zero measure-theoretical entropy.
\end{corollary}
The corollary follows immediately: see, e.g. \cite{Bergelson-Gorodnik-2004} to get absence of weak mixing and \cite{Zimmer-1976} to get zero measure-theoretical entropy. In the case of rational integers Corollary \ref{main-cor} was also proven by Sarnak \cite{Sarnak-Mobius-lectures}.

Corollary \ref{main-cor} suggests that any randomness in the M\"obius function comes from the distribution of $\pm 1$'s, and not from the locations of zeros. 
In the context of rational integers this is expressed by a generalization of (\ref{sum-mu-PNT}):
\begin{conjecture}[Chowla] For every $n_1,\ldots, n_r\in\N$ and $k_1,\ldots, k_r\in\{1,2\}$ not all even
\[\sum_{n=1}^N\mu^{k_1}(n+n_1)\mu^{k_2}(n+n_2)\cdots\mu^{k_r}(n+n_r)=o(N)\]
as $N\to\infty$.
\end{conjecture}
This conjecture, whose only proven instance is  \eqref{sum-mu-PNT}, implies a recent conjecture by Sarnak:
\begin{conjecture}[Sarnak, \cite{Sarnak-Mobius-lectures}]\label{conj-Sarnak} Let $(X,T)$ be a compact topological dynamical system with zero topological entropy. Let $\xi(n)=f(T^n x)$, where $x\in X$ and $f\in C(X,\C)$. Then the sequence $\{\xi(n)\}_{n\geq 1}$ does not correlate with the M\"obius function, i.e.
\[\sum_{n=1}^N\mu(n)\xi(n)=o(N)
\]
as $N\to\infty$.
\end{conjecture}
Sequences $\{\xi(n)\}_{n\geq 1}$ as above are called \emph{deterministic}. It is known that Conjecture \ref{conj-Sarnak} holds true for a wide class of deterministic sequences (see, e.g., \cite{Sarnak-Liu} and references therein).

It is worthwhile to stress the link between topological and measure-theoretical dynamics. For the case of $K=\Q$, it is known that the topological subshift of infinite type obtained by orbit closure of $\{\mu^2(n)\}_{n\geq 1}$ inside $\{0,1\}^{\Z}$ has positive topological entropy $\frac{6}{\pi^2}\log2$ (see \cite{Sarnak-Mobius-lectures}) and, by the variational  principle, one can find invariant probability measures with smaller measure-theoretical entropy. The measure $\Pi$ defined in \eqref{eq:measurePi}, which Sinai and the first author consider in \cite{Cellarosi-Sinai-JEMS} and for which $\{\mu^2(n)\}_{n\geq1}$ is generic, has zero entropy and is the Pinsker factor (largest zero entropy factor) of the measure of maximal entropy. In fact, Peckner \cite{Peckner-2012} showed that the measure of maximal entropy is a Bernoulli extension of $\Pi$. This means that the subshift of infinite type in Theorem \ref{thm-1-intro} is, at least in the case $K=\Q$, a building block for other relevant systems.
It would be of interest to extend this result to arbitrary number fields $K$, where the strictly 1-dimensional method of \cite{Peckner-2012} cannot be applied directly.

Section \ref{sec:examples} illustrates the results in the case of square-free Gaussian integers. Some background on ideals in $\roi$ and group actions with pure point spectrum is given in Section \ref{sec:background}, which may be skipped by readers familiar with these topics. In Section \ref{sec2-mirsky-like}, we show that limits of correlations exist for the sequence $\{\mu^{(k)}(a)\}_{a\in\roi}$. 
In Section \ref{Lambda-and-the-action-on-G}, we construct the spectral measure for the $\roi$-shift and an abstract dynamical system having this spectral measure. 
Section \ref{sec3} contains computations that are used in Section \ref{sec-spectrum-of-action} to prove Theorem \ref{thm-1-intro} and its more detailed version Theorem \ref{main-thm-second-version}. 


\section*{Acknowledgments}
We 
would like to thank 
Yakov G. Sinai for his encouragement and advice, Ali Altu\u{g} for  helpful suggestions, Michael Baake for bringing the references \cite{Baake-Moody-Pleasants-2000, Pleasants-Huck} to our attention, and Dmitry Kleinbock  and the referee for suggestions on improving the text. 
Special appreciation goes to Idris Assani, organizer of the \emph{Ergodic Theory Workshop} in Chapel Hill NC, where part of this work was carried out. 
The research leading to these results has received funding from the European Research Council under the European Union's Seventh Framework Programme (FP/2007--2013) / ERC Grant Agreement n.\ 291147.

\section{An Example: Square-Free Gaussian Integers}\label{sec:examples}
Let $K=\Q(i)$ and $k=2$. Then $\roi=\Z[i]$ is given by Gaussian integers. Since $\Z[i]$ is a principal ideal domain, every ideal is of the form $(a+bi)$ for some $a,b\in\Z$. Units in $\roi$ are $\pm1,\pm i$. The algebraic norm is given by $\an{(a+bi)}=a^2+b^2$. 
Prime ideals $\fp$ are of the form
\begin{itemize}
\item $(p)$, where $p\in\Z$ is a usual prime and $p\equiv3 \bmod4$,
\item $(a+bi)$, such that $a^2+b^2\in\Z$ is a usual prime (necessarily equal to 2 or congruent to $1 \bmod 4$).
\end{itemize}
The first few prime ideals $\fp$ (ordered by norm $N(\fp)$) are $(1+i), (1+2i), (1-2i), (3), (2+3i), (2-3i), \ldots$. 
Square-free Gaussian integers are shown Figure \ref{fig-squarefree-Gaussian}.

%
%
%
%
%
\begin{figure}[ht!]
\begin{center}
\includegraphics[width=14cm]{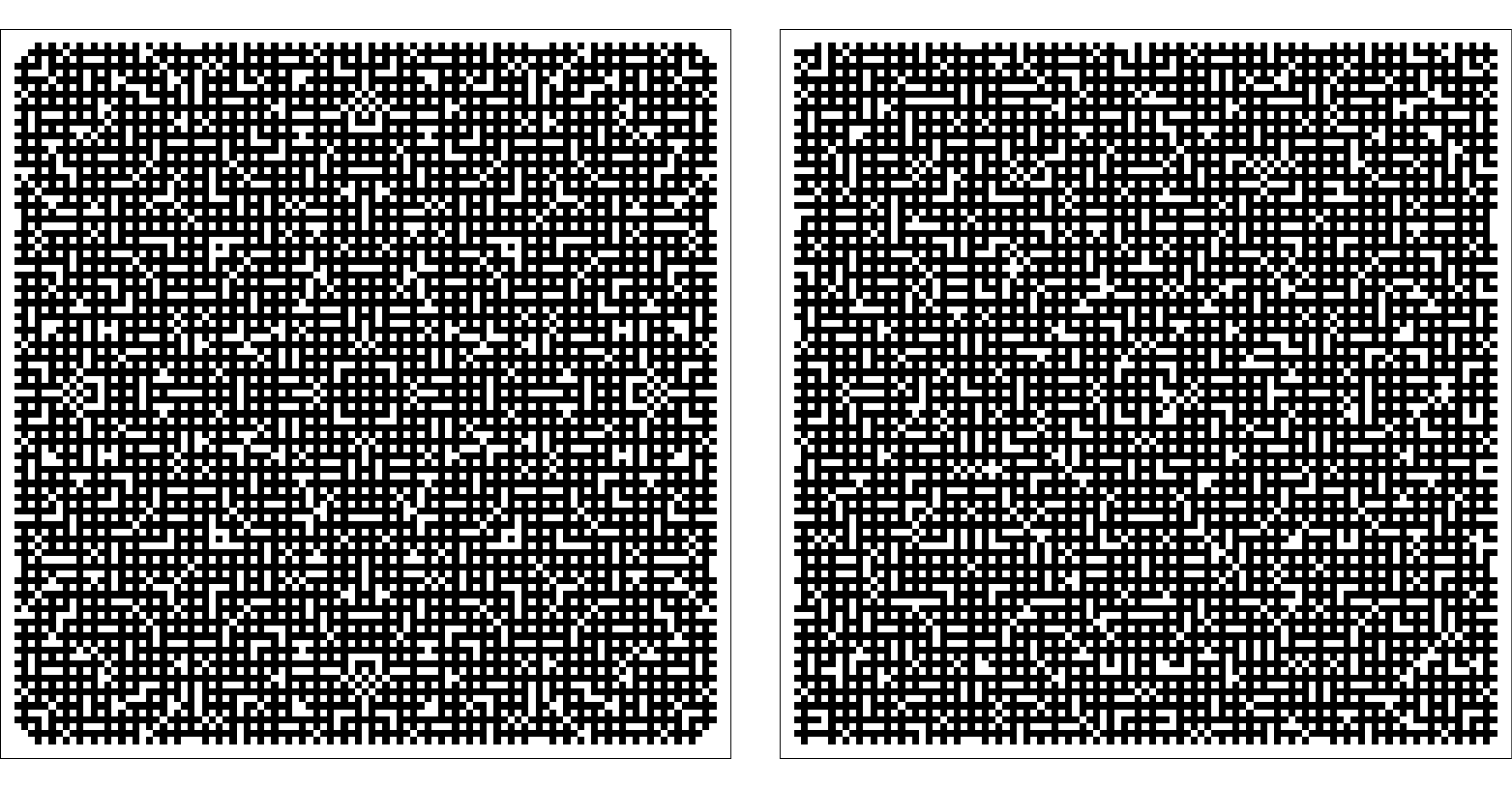}
\caption{{Square-free Gaussian integers. The square grid is identified with $\Z[i]\cong \Z^2$ and each square is colored black (resp. white) if it corresponds to a square-free (resp. not square-free) Gaussian integer. On the left: $\mathcal F_2\cap\{a+bi\}$ where $-50\leq a,b\leq 50$. Notice the dihedral $D_4$ symmetry. On the right, $\mathcal F_2\cap\{a+bi\}$ where $10^{12}\leq a\leq 10^{12}+100$ and $10^{15}\leq b\leq 10^{15}+100$.}}\label{fig-squarefree-Gaussian}
\end{center}
\end{figure}

In this case we can write the Dedekind zeta function in terms of primes in $\Z$. For $\Re s>1$ we have
\begin{align}
\zeta_{\Q(i)}(s)&=\prod_{\fp}\left(1-N^{-s}(\fp)\right)=\nonumber\\
&=(1-2^{-s})^{-1}\prod_{p\equiv1\bmod4}(1-p^{-s})^{-2}\prod_{p\equiv3\bmod4}(1-p^{-2s})^{-1}=\zeta(s)\beta(s),\nonumber
\end{align}
where $\zeta$ denotes the Riemann zeta function and $\beta$ the Dirichlet beta function, 
\[\beta(s)=\sum_{n=0}^\infty\frac{(-1)^n}{(2n+1)^s}.\] 
The density of square-free Gaussian integers is $1/\zeta_{\Q(i)}(2)=\frac{6}{\pi^2 G}\approx0.6637$, where $G=\beta(2)$ is the Catalan constant.
%
%
%

Let us look at correlation functions. For example,
$c_4(1,i,1+i)=0$ because for every $a\in\Z[i]$ at least one of the four Gaussian integers $a$, $a+1$, $a+i$, $a+1+i$, is not square-free, since it is divisible by $2=(1+i)^2$. 
We will show in Proposition \ref{prop-Mirsky-like-formula} that
\[
c_{r+1}(a_1,\ldots,a_r)=\prod_{\fp}\left(1-\frac{D(\fp^2 \mid 0,a_1,\dots,a_r)}{\an{\fp^2}}\right),
\]
where 
$D(\fp^2 \mid 0,a_1,\dots,a_r)$ is the number of \emph{distinct} residue classes among $0+\fp^2, a_1+\fp^2,\ldots, a_r+\fp^2$ in $\Z[i]/\fp^2$.
%
%
The fact that $c_4(1,i,1+i)=0$ can be derived by the formula above. In fact, notice that $D(\fp^2\mid0,1,i,1+i)=4$ 
for all prime ideals $\fp$ and that there is a prime ideal, $\fp_1=(1+i)$, for which $\an{\fp_1^2}=4$.

On the other hand, $c_5(1,i,-1,-i)=\left(1-\frac{3}{4}\right)\prod_{\fp\neq\fp_1}(1-\frac{5}{\an{\fp^2}})>0$. In fact, among $
0+\fp_1^2,1+\fp_1^2,i+\fp_1^2,-1+\fp_1^2,-i+\fp_1^2$ there are only three distinct residue classes, 
whilst for every prime ideal $\fp\neq\fp_1$ we have $D(\fp^2\mid0,1,i,-1,-i)=5$ 
and $\an{\fp}\geq 5$. More precisely 
\[c_5(1,i,-1,-i)=\frac{1}{4}\prod_{p\equiv1\bmod 4}\!\left(1-\frac{5}{p^2}\right)^2\prod_{p\equiv3\bmod4}\!\left(1-\frac{5}{p^4}\right)\approx0.1303.\]


Let us identify $\Z^2$ and $\Z[i]$, via $\iota\colon (a,b)\mapsto a+bi$.
The $\Z[i]$-action on $(\{0,1\}^{\Z[i]},\Pi)$ is simply given by the two commuting translations $a+ib\mapsto (a+1)+bi$ and $a+bi\mapsto a+(b+1)i$, under which the probability measure $\Pi$ is invariant. By construction, the 2-dimensional sequence $\{\mu^2(a+ b i)\}_{a+b i\in\Z[i]}$ is generic for this action.
Consider now the group 
\[\G=\prod_\fp \Z[i]/\fp^2.\]
It is the direct product of finite abelian groups $\Z[i]/(1+i)^2$, $\Z[i]/(1+2i)^2$, $\Z[i]/(1-2i)^2$, $\Z[i]/(3)^2$, $\Z[i]/(2+3i)^2$, $\Z[i]/(2-3i)^2,\ldots$ 
and it is acted upon coordinate-wise by $\Z^2$ via $\iota$.

 \begin{figure}[ht!]
\begin{center}
\vspace{-.3cm}
\includegraphics[width=13.75cm]{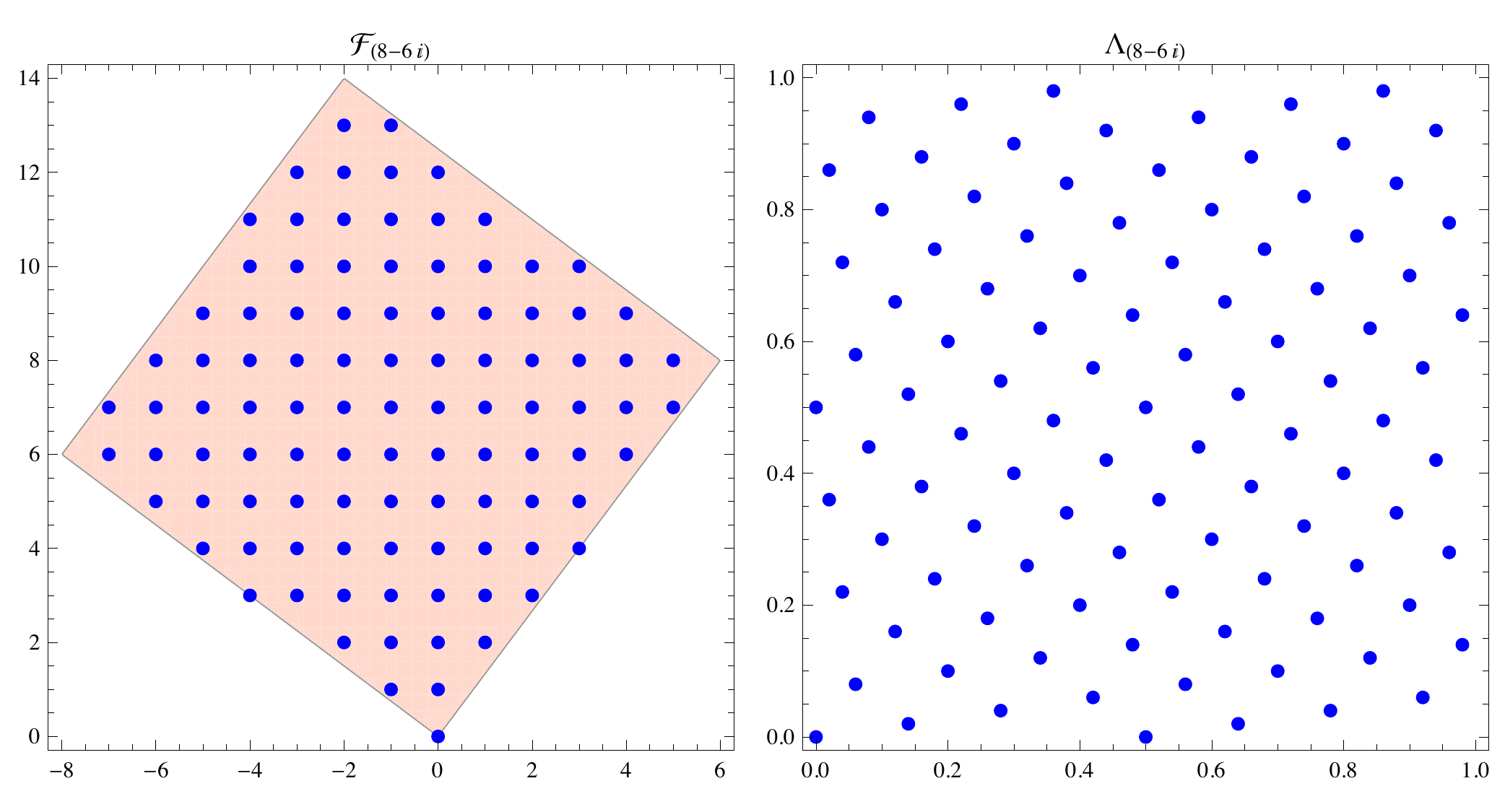}
\caption{On the left: the fundamental domain for the ideal $(8-6i)=(1+i)^2(1-2i)^2$, given by the square of vertices $0, 6+8i, -8+6i, -2+14i$. It contains $N(8-6i)=2^25^2=100$ Gaussian integers. On the right: the annihilator $(8-6i)^\perp$, identified with of a subset of rational points of the 2-torus whose coordinates have denominator $N(8-6i)=100$.}\label{fig-exFLambda}
\end{center}
\end{figure}

The Main Theorem \ref{main-thm-second-version} states that the two actions $\Z[i] \curvearrowright(\{0,1\}^{\Z[i]},\Pi)$ and $\Z^2 \curvearrowright (\G,\mathrm{Haar})$ are isomorphic. More precisely, they have pure point spectrum given by the discrete group $\Lambda=\hat{\G}$,
identified via $\iota$ with a subset of $\T^2$ (viewed as $\widehat {\Z^2}$). 


Here is the explicit construction of $\Lambda$ in this example.
For every square-free ideal $\fd\subseteq \Z[i]$, view $\fd$ as subgroup of $\Z^2$ and consider a fundamental domain $\mathfrak F_{\fd^2}$ for $\Z^2/\fd^2$, say the square with sides $w=(w_1,w_2)$, and $w'=(-w_2,w_1)$, where $w_1>0$, $w_2\geq0$, and $w_1^2+w_2^2=N(\fd)$. In this way $\#\mathfrak F_{\fd^2}=N(\fd^2)$. One can check that the annihilator 
of the ideal $\fd^2$ in $\widehat{\Z[i]}$ can be identified with a subset of $\widehat{\Z^2}=\T^2$ and written as $$(\fd^2)^\perp
=\frac{1}{N(\fd^2)}\left(\begin{array}{cc}-w_2 & w_1 \\w_1 & w_2\end{array}\right)\mathfrak F_{\fd^2}\subseteq \left\{\left(\frac{t_1}{N(\fd^2)},\frac{t_2}{N(\fd^2)}\right)\in\T^2,\:0\leq t_1,t_2<N(\fd^2)\right\}.$$
See Figure \ref{fig-exFLambda} for an example.

The spectrum $\Lambda$ is the subgroup of $\T^2$ obtained as union of the $(\fd^2)^\perp
$'s as above, and is shown in Figure \ref{spectrumLambda}.
\newpage
\begin{figure}[h!]
\begin{center}
\hspace{-.3cm}
\includegraphics[width=16.75cm]{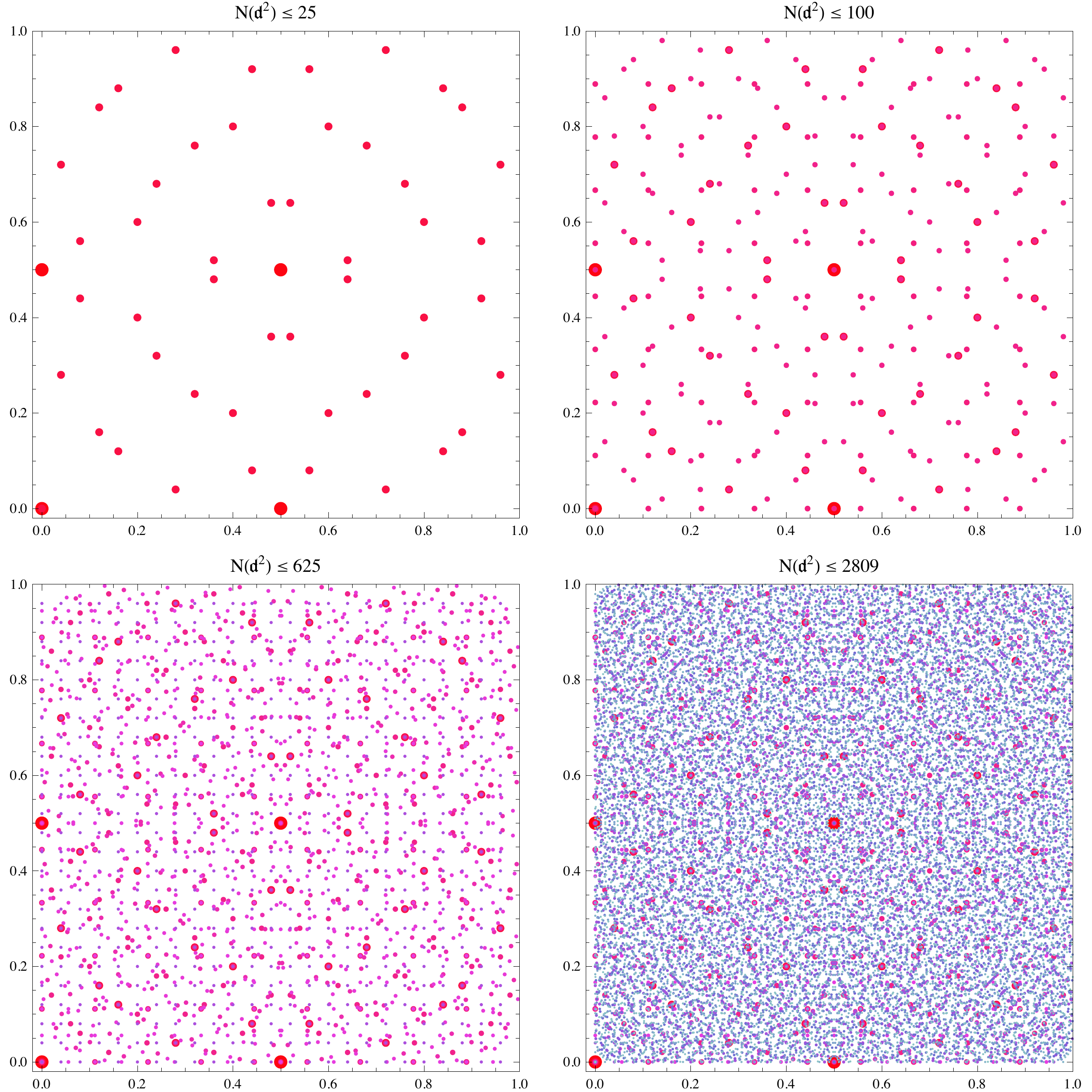}
\caption{Successive approximations $\bigcup_{N(\fd^2)\leq D}(\fd^2)^\perp$ of the spectrum $\Lambda$, with $D=25, 100, 625, 2809$. The size 
of points in $(\fd^2)^\perp$ decreases as $N(\fd^2)$ increases.}\label{spectrumLambda}
\end{center}
\end{figure}

\section{Background}\label{sec:background}
\subsection{Ideals in Number Fields}
Let $K$ be a number field of degree $[K:\Q]=d$, and let $\roi$ be its ring of integers. While $\roi$ need not be a principal ideal domain, it is always a Dedekind domain, that is an integral domain whose proper ideals factor (uniquely, up to the order) into a product of prime ideals (see \cite{Narkiewicz, Neukirch}). We will denote ideals in $\roi$ by $\fa, \mathfrak b, \mathfrak c, \mathfrak d, \mathfrak n$, and prime ideals by $\fp, \mathfrak q$. The sum of two ideals is defined as $\fa+\mathfrak b=\{a+b\colon a\in\fa,\:b\in\mathfrak b\}$, the product as $\fa\mathfrak b=\{\sum_{i=1}^na_ib_i,\: a_i\in\fa,\: b_i\in\mathfrak b,\:n\in\N \}$, and these are ideals.
We say that \emph{$\fa$ divides $\mathfrak b$} 
if and only if $\mathfrak b=\fa \mathfrak c$ for some ideal $\mathfrak c$ or, equivalently, if and only if $\fa\supseteq \mathfrak b$.
For any two ideals $\fa, \mathfrak b$ we define their \emph{greatest common divisor} as the smallest ideal containing both $\fa$ and $\mathfrak b$, that is $\gcd(\fa, \mathfrak b)=\fa+\mathfrak b$. The \emph{least common multiple} of $\fa$ and $\mathfrak b$ is defined as the largest ideal contained in both $\fa$ and $\mathfrak b$, that is $\lcm(\fa, \mathfrak b)=\fa\cap\mathfrak b$. 
Let $k\geq2$. 
The indicator $\mu^{(k)}$ of $k$-free ideals satisfies
%
\be
\mu^{(k)}(\fa)=\sum_{\fb^k\supseteq \fa}\mu(\fb).\label{mu-m-mu}
\ee

We will say that two integers $a,b\in\roi$ are \emph{congruent modulo the ideal $\fa$} (denoted by $a\equiv b \bmod\fa$) if and only if $a+\fa=b+\fa$ in the finite ring $\roi/\fa$. 
The \emph{algebraic norm 
} of a nonzero ideal $\fa$ is defined as $\an{\fa}=\#\roi/\fa=[\roi:\fa]$, where $\#\cdot$ denotes the cardinality of a finite set. There is also a notion of \emph{norm} for elements of $K$. Let $a\in K$ and let $m_a\colon K\to K$ be the $\Q$-linear map $m_a(b)=ab$. The norm $N_{K/\Q}(a)$ is defined as the determinant of $m_a$. If $K/\Q$ is Galois, then $N_{K/\Q}$ equals the product of the conjugates of $a$. Note that $N_{K/\Q}(a)$ need not be positive. When $\roi$ is a principal ideal domain, and $\fa=(a)$, then $\an{\fa}=N_{K/\Q}(a)$.
The \emph{Dedekind zeta function} is given, for $\Re s>1$,  by
\be\zeta_K(s)=\sum_{\fa}\an{\fa}^{-s}=\prod_{\fp}\left(1-\an{\fp}^{-s}\right)^{-1},\label{eq:zeta}
\ee
where the sum ranges over all nonzero ideals of $\roi$ and the product over the prime ones. Let us also set $\an{a}=\an{(a)}$.

\subsection{Group Actions with Pure Point Spectrum}\label{sec:purepoint}
Let us recall the notions of ergodic and pure point spectrum actions relevant to our setting. Let $G$ be  a separable locally compact abelian
group and let $S$ be a standard Borel $G$-space. Let $\mu$ be a $\sigma$-finite (left) $G$-invariant measure (that is $\mu(x E)=\mu(E)$ for all $x\in G$ and all Borel subset $E$ of $S$). One says that the action $G\curvearrowright(S,\mu)$ is \emph{ergodic} if whenever $E$ is a Borel set in $S$ with $\mu(E)\neq0\neq\mu(E^c)$, then for some $x\in G$ we have $\mu(E\vartriangle xE)>0$. Let $U$ be the unitary representation of $G$ on $L^2(S,\mu)$ given by $(U_xf)(s)=f(xs)$. When $U$ is a discrete direct sum of finite-dimensional irreducible representations one says that the action $G\curvearrowright(S,\mu)$ has \emph{pure point spectrum}. This means that there exists an orthonormal 
basis $\{\Phi_j\}_{j=1}^\infty$ for $L^2(S,\mu)$ and a countable subgroup $\Gamma=\{\chi_j\}_{j=1}^\infty$ of $\hat G$ such that $U_x\Phi_j=\chi_j(x)\Phi_j$. The group $\Gamma$ is referred to as the \emph{spectrum} of the action.

One can construct ergodic actions with pure point spectrum as follows. Let $K$ be a compact group, $H\subseteq K$  a closed subgroup, $\varphi$ a continuous homomorphism of $G$ onto a dense subgroup of $K$. Define the action $G\curvearrowright(K/H,\mu_H)$ by $x(k+H)=\varphi(x)k+H$, where $\mu_H$ is the unique measure on  $K/H$ such that $\mu_H(K/H)=1$, $G$-invariant under left multiplication. One can check that the latter is ergodic (by transitivity of the action) and has pure point spectrum (the unitary representation $U$ of $G$ on $L^2(K/H,\mu_H)$ is given by $U_x=V_{\varphi(x)}$, where $V$ is the unitary representation of $K$ on $L^2(K/H,\mu_H)$; since $V$ is a subrepresentation of the regular representation of $K$, it decomposes into a direct sum of finite-dimensional irreducible representations by the Peter-Weyl Theorem). Mackey \cite{Mackey-1964} proved that, modulo removing null sets in $S$ and $K/H$, every ergodic action with pure point spectrum can be realized as above. His work generalized the 
classical theory by von Neumann \cite{vonNeumann-1932}  and Halmos and von Neumann \cite{Halmos-vonNeumann-1942} where $G=\Z$. In all these cases, for actions with pure point spectrum, the isomorphism class is uniquely determined by its spectrum.

\section{Arithmetical Pattern Problems for $k$-Free Ideals}\label{sec2-mirsky-like}




%
%
%


We need a notion of size on \roi\ with the property that any ball of finite radius is finite. This is in general not true for the algebraic norm $N$, as there are number fields whose group of units is infinite.
To avoid this problem, we consider a \emph{geometric norm} on $\|\cdot\|$ by viewing $\roi$ as a vector space over $\Q$.
Let us fix $d$ generators for $\roi$, i.e.\ elements $e_1,\ldots,e_d\in\roi$ such that $\Z[e_1,\dots,e_d]=\roi$, and thus define the 
isomorphism $\iota\colon (\Z^d,+)\to(\roi,+)$, $(\alpha_1,\ldots,\alpha_d)\mapsto a=\alpha_1e_1+\ldots+\alpha_d e_d$. 
For an element $a\in \roi$ let $\| a\|=|\alpha_1|+\ldots+|\alpha_d|$ 
be the $L^1$ norm induced from $\Z^d$. Our results do not 
depend on the choice of $\iota$, except for implied constants in error terms. 


Let $B_x=\{a\in\roi \colon \|a\|\leq x\}$ denote the ball or radius $x$ in $\roi$, with respect to this geometric norm. 
Suppose $\mathfrak a = \langle \sum_{j} k_{ij} e_j\rangle_{i=1}^d=\langle  v_i\rangle_{i=1}^d.$ For any choice of generators we call the set $\Delta=\{\sum_j \eps_{ij}v_j\colon \eps_{ij}=0\text{ or }1\}$ a \emph{cell} of $\mathfrak a$. The \emph{diameter} of an ideal $\mathfrak a$, written $\diam \mathfrak a$, is defined by $$\min_{\Delta \text{ is a cell of }\mathfrak a}\left(\max_{a\in\Delta}\| a\|\right).$$ 


We introduce several abbreviations to simplify notation. We write $\underline{a}=(a_1,\dots,a_s)\in (\roi)^s$, $\underline{\mathfrak n}=(\mathfrak n_1,\dots,\mathfrak n_s)$ for $s$-tuples of ideals of \roi, $\underline{\mathfrak n}^k=(\mathfrak n_1^k,\dots,\mathfrak n_s^k)$ for $s$-tuples of powers of ideals, and $\mu(\underline{\mathfrak n})=\prod_{i=1}^s\mu(\mathfrak n_i)$ for the product M\"obius function. We also tacitly set the range for integers $i$ and $j$ to be $\{1,\dots,s\}$. Define the function $D$ by 
\beq D(\mathfrak a\mid \underline a)=\#\{b    \bmod \mathfrak a\mid b\equiv a_i    \bmod \mathfrak a \text{ for some }i\}\label{eq: D}\eeq
and more generally set 
\beq 
D(\mathfrak a,\mathfrak b\mid \underline a)=
\#\{b    \bmod  \mathfrak a\mid b\equiv a_i    \bmod \mathfrak a \text{ for some }i \text{ and }b \in\gcd(\mathfrak a,\mathfrak b)\}.
\eeq

\begin{prop}[Existence of correlation functions]\label{prop-Mirsky-like-formula}
For every $r\geq1$ and every  $a_1,\ldots,a_r \in\roi$, the limit 
\beq
c_{r+1}(a_1,\ldots,a_r)=\lim_{x\to\infty}\frac{1}{\# B_x}\sum_{a\in B_x}\mu^{(k)}(a)\mu^{(k)}(a+a_1)\cdots\mu^{(k)}(a+a_r)
\label{limit-c_r-quote}
\eeq
exists and
\beq
c_{r+1}(a_1,\ldots, a_r)=\prod_{\fp}\left(1-\frac{D(\fp^k \mid 0,a_1,\dots,a_r)}{\an{\fp^k}}\right),\label{general-c_r+1-Mirsky}
\eeq
where $D$ is as in equation \eqref{eq: D}.
\end{prop}
We shall refer to $c_{r+1}$ as the \emph{$(r+1)$-st correlation function for the set of $k$-free integers in \roi}; we will not indicate the dependence on $k$ explicitly. 
By taking $r=1$ and $a_1=0$ in Proposition \ref{prop-Mirsky-like-formula}, we have the 
well known
\begin{corollary}\label{cor-density-k-free}  
The asymptotic density of $k$-free integers in $\roi$ is 
\beq
c_2(0)=
\lim_{x\to\infty}\frac{1}{\#B_x}\sum_{a\in B_x}\mu^{(k)}(a)=
\prod_{\fp}\left(1-\frac{1}{\an{\fp^k}}\right)=\frac{1}{\zeta_K(k)},\label{density-squarefree}
\eeq
where $\zeta_K$ is as in \eqref{eq:zeta}.
\end{corollary}

We will actually prove a more general version of Proposition \ref{prop-Mirsky-like-formula}, namely a quantitative asymptotic statement on the frequencies of arbitrary binary configurations in $\{\mu^{(k)}(a)\}_{a\in\roi}$ with an additional divisibility constraint (Theorem \ref{th:M=S+O} below). Set 
\beq M_k(x;\mathfrak b; \underline a)=\sum_{\substack{{a\in \mathfrak b}\\{\| a\|\le x}}} \mu^{(k)}(a+a_1)\dots \mu^{(k)}(a+a_s). 
\eeq


\begin{theorem}\label{th:M=S+O}
If $D(\mathfrak p^k,\mathfrak b\mid \underline a)=\frac{N(\mathfrak p^k)}{N(\gcd(\mathfrak p^k,\mathfrak b))}$ 
for some $\mathfrak p$, then $$M_k(x;\mathfrak b;\underline a)=0.$$ Otherwise 
\beq M_k(x;\mathfrak b;\underline a)=S_{k,\mathfrak b}(\underline a) x^d+O_{\iota, \eps}(x^{d-\frac{k-1}{k+2sk-1}+\eps}) \eeq 
for positive $S_{k,\mathfrak b}(\underline a)$ computed in \eqref{eq:S} and every $\eps>0$. 
\end{theorem}

For ideals $\mathfrak n_1,\dots,\mathfrak n_s$ define the $E$ symbol by \beq E\binom{\underline{\mathfrak n}}{\underline{  a}}
=\begin{cases}1&\text{ if there exists $b$  s.t. } b+a_i\equiv 0\bmod  \mathfrak n_i \text{ for all } i\\0&\text{ otherwise.}\end{cases}\label{eq:chinese1}\eeq

\begin{lemma}\label{lem:mirsky1}
Equation \eqref{eq:chinese1} evaluates to 1 precisely when $$a_i-a_j\equiv0\bmod  \gcd(\mathfrak n_i,\mathfrak n_j)$$ for every $i$ and $j$. In this case there is exactly one $b$ in each residue class modulo $\lcm ( \underline{\mathfrak n}).$
\end{lemma}

\begin{proof}It is enough to observe that 
\begin{align*}
\roi/\lcm(\underline{\mathfrak n})&\to \roi/\mathfrak n_1\times\dots\times\roi/\mathfrak n_s\\
a+\lcm(\underline{\mathfrak n})&\mapsto (a+\mathfrak n_1,\dots,a+\mathfrak n_s)
\end{align*} is an isomorphism onto its image.
\end{proof}

\begin{lemma}\label{lem:mirsky2}
Let $$T(x)=\#\{b \in\roi\colon \| b\|\le x \text{ and } b+a_i\equiv 0\bmod  \mathfrak n_i \text{ for all }i\}$$ with notation as before. Then 
\[ \left| T(x)-E\binom{\underline{\mathfrak n}}{\underline{  a}}\frac{x^d }{N(\lcm(\underline{ \mathfrak n}))}\right|=O\left(x^{d-1}\diam\lcm (\underline{\mathfrak n})\right).\]
\end{lemma}

\begin{proof}We omit the proof as it is standard. 

\end{proof}

\begin{lemma}\label{lem:mirsky3}
Let $T(x)$ be the number of solutions to the system 
\[\left\{\begin{aligned}b_i-b&=a_i\text{ for all }i\\ b_i&\in \mathfrak n_i\\ b&\in\mathfrak b\end{aligned}\right.\]
such that $ \|b\|\le x$. Then we have 
\[ \left|T(x)-x^d\frac{ E\begin{pmatrix}\mathfrak b&\underline{\mathfrak n}\\0&\underline{a}\end{pmatrix}}{N(\lcm(\mathfrak b, \underline{\mathfrak n}))}\right|=O\left(x^{d-1}\diam\lcm (\mathfrak b,\underline{\mathfrak n})\right).\]
\end{lemma}

\begin{proof}Apply Lemmata \ref{lem:mirsky1} and \ref{lem:mirsky2} with $s$ replaced by $s+1$. 
\end{proof}

\begin{lemma}\label{lem:mirsky4}
If a function is multiplicative on ideals, then it is determined by its values at prime powers. That is, if 
\[\sum_{\underline{\mathfrak n}}|f(\underline{\mathfrak n})|<\infty\] 
and 
\[f(\underline{\mathfrak n})f(\underline{\mathfrak n}')=f(\mathfrak n_1\mathfrak n'_1,\dots,\mathfrak n_s\mathfrak n'_s)\text{ whenever }\gcd(\mathfrak n_i,\mathfrak n'_j)=\roi \text{ for all }i,j,\] 
then 
\[\sum_{\underline{\mathfrak n}}f(\underline{\mathfrak n})=\prod_{\mathfrak p}\chi_{\mathfrak p}\] 
where 
\[\chi_{\mathfrak p}=\sum_{\delta_1,\dots,\delta_s\ge 0}f(\mathfrak p^{\delta_1},\dots,\mathfrak p^{\delta_s}).\]
\end{lemma}

\begin{proof}
Follows by a simple induction.
\end{proof}

\begin{lemma}\label{lem:mirsky5}
Take $t\ge 1$ and $m\ge 2$. Also fix a prime ideal $\mathfrak p$ and $k_1,\dots,k_t\in\roi$ such that \[k_1\equiv\dots\equiv k_t\equiv \rho\bmod  \mathfrak p^k.\] Then 
\beq \sum_{\substack{{\eta_1,\dots,\eta_t\in\{0,1\}}\\{\text{not all zero}}}}(-1)^{\eta_1+\dots+\eta_t} E\begin{pmatrix}\mathfrak b&\mathfrak p^{k\eta_1}&\dots&\mathfrak p^{k\eta_t}\\0&k_1&\dots&k_t\end{pmatrix}=\begin{cases}-1&\text{ if }\rho\in\gcd(\mathfrak p^k,\mathfrak b)\\0&\text{ otherwise.}\end{cases} \eeq
\end{lemma}

\begin{proof}
Note that $\eta_i$ is non-zero for at least one $i$, so that 
\[E
\begin{pmatrix}\mathfrak b&\mathfrak p^{k\eta_1}&\dots&\mathfrak p^{k\eta_t}\\
0&k_1&\dots&k_t
\end{pmatrix}
=
E
\begin{pmatrix}
 \mathfrak b& \mathfrak p^k\\
0&\rho
\end{pmatrix}
=
\begin{cases}
 1&\rho\in\gcd(\mathfrak p^k,\mathfrak b)\\
0&\text{otherwise}
\end{cases}
\] by Lemma \ref{lem:mirsky1}. Furthermore 
\[\sum_{\substack{{\eta_1,\dots,\eta_t\in\{0,1\}}\\{\text{not all zero}}}}(-1)^{\eta_1+\dots+\eta_t} =-1,\]
whence the result.
\end{proof}

\begin{lemma}\label{lem:mirsky6}
With notation as before set 
\beq S=S_{k,\mathfrak b}(\underline{a})=\sum_{\underline{\mathfrak n}}\frac{\mu(\underline{\mathfrak n})}{N(\lcm(\mathfrak b,\underline{\mathfrak n}^k))} 
E\begin{pmatrix}\mathfrak b&\underline{\mathfrak n}^k\\0&\underline{a}\end{pmatrix}. \label{eq:S}\eeq Then we have 
\beq 
S=\dfrac1{N(\mathfrak b)}\displaystyle\prod_{\mathfrak p^k\nsupseteq \mathfrak b}\left(1-\frac{N(\gcd(\mathfrak p^k,\mathfrak b))  D(\mathfrak p^k,\mathfrak b\mid \underline a)}
{ N(\mathfrak p^k)  }\right), \eeq
and it vanishes precisely when 
$$D(\mathfrak p^k,\mathfrak b\mid \underline a)=\frac{N(\mathfrak p^k)  }{ N(\gcd(\mathfrak p^k,\mathfrak b))}\text{ for some }\mathfrak p.$$
\end{lemma}

\begin{proof}
We claim that if \beq\label{eq:rel_primality}\gcd(\mathfrak n_i,\mathfrak n'_j)=\roi\eeq for all $i$ and $j$, then 
\beq \label{eq:first_claim}
\frac1{N(\mathfrak b)}N(\lcm(\mathfrak b,\underline{\mathfrak n}))\cdot\frac1{N(\mathfrak b)}N(\lcm(\mathfrak b,\underline{\mathfrak n}'))=
\frac1{N(\mathfrak b)}N(\lcm(\mathfrak b,\mathfrak n_1\mathfrak n'_1,\dots,\mathfrak n_s\mathfrak n'_s)) \eeq and 
\beq \label{eq:second_claim}E\begin{pmatrix}\mathfrak b&\underline{\mathfrak n}\\0&\underline{a}\end{pmatrix}\cdot E\begin{pmatrix}\mathfrak b&\underline{\mathfrak n}'\\
0&\underline{a}\end{pmatrix}=E\begin{pmatrix}\mathfrak b&\mathfrak n_1\mathfrak n'_1&\dots&\mathfrak n_s\mathfrak n'_s\\0&a_1&\dots&a_s\end{pmatrix}.\eeq
For each prime ideal $\mathfrak p$ let $\mathfrak p^\lambda$, $\mathfrak p^{\nu_i}$, $\mathfrak p^{\nu'_j}$ be its largest powers that divide $\mathfrak b$, $\mathfrak n_i$, and $\mathfrak n'_j$, respectively. To prove \eqref{eq:first_claim} it is enough to confirm that 
$$\max(\lambda,\nu_1,\dots,\nu_s)+\max(\lambda,\nu'_1,\dots,\nu'_s)-2\lambda=\max(\lambda,\nu_1+\nu'_1,\dots,\nu_s+\nu'_s)-\lambda.$$ From \eqref{eq:rel_primality} it follows that $\nu_i=0$ for all $i$ or $\nu'_j=0$ for all $j$, so the preceding equation is verified.

For the second claim \eqref{eq:second_claim} note that if the left-hand side vanishes, then so does the right-hand side. So suppose the left-hand side doesn't vanish, that is, 
\begin{align*}
 a_i-a_j\in\gcd(\mathfrak n_i,\mathfrak n_j), &\quad a_i\in\gcd(\mathfrak b,\mathfrak n_i)\\
 a_i-a_j\in\gcd(\mathfrak n'_i,\mathfrak n'_j),&\quad a_i\in\gcd(\mathfrak b,\mathfrak n'_i)
\end{align*}
by Lemma \ref{lem:mirsky1}. From \eqref{eq:rel_primality} these conditions are equivalent to 
$$a_i-a_j\in\gcd(\mathfrak n_i\mathfrak n'_i,\mathfrak n_j\mathfrak n'_j), \quad a_i\in\gcd(\mathfrak b,\mathfrak n_i\mathfrak n'_i),$$
and the claim follows by another application of Lemma \ref{lem:mirsky1}.

Now then we write 
\beq
N(\mathfrak b) S=\sum_{\underline{\mathfrak n}} \frac{\mu(\underline{\mathfrak n})N(\mathfrak b)}{N(\lcm(\mathfrak b,\underline{\mathfrak n}^k))} E\begin{pmatrix}\mathfrak b& \underline{\mathfrak n}^k\\0&\underline a\end{pmatrix}=
\sum_{\underline{\mathfrak n}} f(\underline{\mathfrak n}). \eeq
This sum converges absolutely 
\[\sum_{\underline{\mathfrak n}}|f(\underline{\mathfrak n})|\le N(\mathfrak b)\sum \frac1{N(\lcm(\underline{\mathfrak n}^k)}\ll \sum_{\underline{\mathfrak n}}\frac1{N^k(\mathfrak n)}\sum_{\lcm(\underline{\mathfrak n})=\mathfrak n}1
\ll \sum_{\mathfrak n}\frac{d^s(\mathfrak n)}{N^k(\mathfrak n)}\ll \zeta_K(k-\eps)<\infty
\]
since $k\ge 2$ and the divisor function satisfies $d(\mathfrak n)\ll_\eps N^\eps(\mathfrak n).$ By multiplicativity and Lemma \ref{lem:mirsky4} we have 
$$N(\mathfrak b)S=\prod_{\mathfrak p}\chi_{\mathfrak p},$$ with 
\begin{align*}\chi_{\mathfrak p}
&=\sum_{\delta_1,\dots,\delta_s\ge 0}\frac{(-1)^{\delta_1+\dots +\delta_s}N(\mathfrak b)}{N(\lcm(\mathfrak b,\mathfrak p^{k\delta_1},\dots,\mathfrak p^{k\delta_s}))} E\begin{pmatrix}\mathfrak b& \mathfrak p^{k\delta_1} &\dots&\mathfrak p^{k\delta_s}\\0&a_1&\dots&a_s\end{pmatrix} \\
&=1+\frac{N(\gcd(\mathfrak p^k,\mathfrak b))}{N(\mathfrak p^k)}\psi_{\mathfrak p} 
\end{align*}
and 
\beq\label{eq:psi}\psi_{\mathfrak p}=\sum_{\substack{{\delta_1,\dots,\delta_s\in\{0,1\}}\\{\text{not all zero}}}} (-1)^{\delta_1+\dots +\delta_s} E\begin{pmatrix}\mathfrak b & \mathfrak p^{k\delta_1}&\dots&\mathfrak p^{k\delta_s}\\0&a_1&\dots&a_s\end{pmatrix}.\eeq 

We evaluate $\psi_{\mathfrak p}$. Observe that the terms of \eqref{eq:psi} are zero if $a_j$ are not congruent modulo $\mathfrak p^k$ for all $j\in\{i\colon \delta_i=1\}$. For $\rho$ modulo $\mathfrak p^k$ let $t_\rho$ denote the number of integers $a_1,\dots,a_s$ that are congruent to $\rho$. If $t_\rho>0$, denote them by $k_1^{(\rho)},\dots,k_{t_\rho}^{(\rho)}.$ Then using Lemma \ref{lem:mirsky5} we have 
\begin{align*}\psi_{\mathfrak p}&=\sum_{\substack{{\rho\bmod \mathfrak p^k}\\{t_\rho>0}}} \sum_{\substack{{\eta_1,\dots,\eta_{t_\rho}\in\{0,1\}}\\{\text{not all zero}}}} 
(-1)^{\eta_1+\dots +\eta_{t_\rho}} E\begin{pmatrix}\mathfrak b&\mathfrak p^{k\eta_1}&\dots &\mathfrak p^{k\eta_{t_\rho}}\\0& k_1^{(\rho)}&\dots & k_{t_\rho}^{(\rho)}\end{pmatrix}= 
\\&=-\sum_{\substack{{\rho\bmod \mathfrak p^k}\\{t_\rho>0}\\{\rho\in\gcd(\mathfrak p^k,\mathfrak b)}}}1=
-D(\mathfrak p^k,\mathfrak b\mid \underline a). \end{align*} Thus we have 
\[N(\mathfrak b) S=\prod_{\mathfrak p}\left(1-\frac{N(\gcd(\mathfrak p^k,\mathfrak b))}{N(\mathfrak p^k)} D(\mathfrak p^k,\mathfrak b\mid \underline a)\right).\]

It remains to verify the positivity part. Clearly the product vanishes if one of the factors does. Otherwise the factors cannot be less than $1-\frac s{N^k(\mathfrak p)}$ for $N(\mathfrak p)$ large enough, and these give a non-zero product. 
\end{proof}

\begin{lemma}\label{lem:product}
For all $x, y$ in $\roi$ we have 
$$\|xy\|\ll_\iota \|x\| \|y\|.$$
\end{lemma}

\begin{proof}
Let $\{e_i\}_{i=1}^d$ be the basis used to define $\iota$. 
Then $e_ie_j=\sum_m c^{ij}_m e_m$ for certain $c^{ij}_m\in \roi$. We have 
\begin{align*}
\|xy\|&=\left\|\sum_i x_ie_i\sum_j y_j e_j\right\| =\left\|\sum_{i,j,m} x_i y_j c^{ij}_m e_m\right\|\\
&\le\max|c^{ij}_m|d \sum_{i,j}|x_i y_j|\ll_\iota \left(\sum_i |x_i|\right)\left(\sum_j |y_j|\right)\\
&\ll_\iota \|x\|\|y\|. 
\end{align*}
\end{proof}

\begin{lemma}\label{lem:added_lemma}
For every $a\in\roi\smallsetminus \{0\}$ we have $N(a) \ll_\iota \|a\|^d.$  
\end{lemma}

\begin{proof}
Fix $b\in\roi$ whose irreducible polynomial has degree $d$ (it exists by the Primitive Element Theorem). Let $\Xi=\Z[a,ab,ab^2,\dots,ab^{d-1}].$ Observe that $\Xi<\roi$ is a finite index additive subgroup and that $\Xi\subseteq (a), $ implying that $\#\roi/\Xi\ge N((a))$. Using Lemma \ref{lem:product} we  have \[
N(a)\le \#\roi/\Xi\le \prod_{i=0}^{d-1} \| ab^i\| \ll_{\iota } \|a\|^d \prod_{i=0}^{d-1}\|b^i\|\ll_{\iota, b} \|a\|^d,
\] as needed. Note that minimizing $\prod_{i=0}^{d-1}\| b^i\|$ over $b$ with full degree irreducible polynomial will remove the dependence on the choice of $b$. 
\end{proof}

\begin{proof}[Proof of Theorem \ref{th:M=S+O}]Say $D(\mathfrak p^k,\mathfrak b\mid \underline a)=\frac{N(\mathfrak p^m)}{N(\gcd(\mathfrak p^k,\mathfrak b))}$ for some $\mathfrak p$. Then it is easy to see that $a+a_i\in\mathfrak p^k$ for some $i$, proving the first case.

Suppose then that $D(\mathfrak p^k,\mathfrak b\mid \underline a)<\frac{N(\mathfrak p^k)}{N(\gcd(\mathfrak p^k,\mathfrak b))}$ for all $\mathfrak p$. Let $\alpha\in (0,1/k)$ to be chosen later. From the relation (\ref{mu-m-mu}) 
we get that
\begin{align*}
M_k(x;\mathfrak b; \underline a)&=\sum_{\underline{\mathfrak n}}\sum_{\substack{{\| a\|\le x}\\{a\in \mathfrak b}\\{a+a_i\in\mathfrak n_i^k}\\{1\le i\le s}}}  
\mu(\underline{\mathfrak n})=
\sum_{\underline{\mathfrak n}}\mu(\underline{\mathfrak n})
\sum_{  \substack{\| a\| \le x\\ a\in \mathfrak b\\a+a_i\in\mathfrak n_i^k }}1= 
\\&=\Sigma_1+\Sigma_2. 
\end{align*}
The first sum is over $s$-tuples of ideals $\mathfrak n_i$ of norm at most $x^\alpha$, while $\Sigma_2$ includes $s$-tuples where at least one ideal has norm greater than $x^\alpha$. By Lemma \ref{lem:mirsky3},
\begin{align}\notag \Sigma_1 &=\sum_{N(\mathfrak n_i)\le x^\alpha}\mu(\underline{\mathfrak n})\left(x^d\frac{E\begin{pmatrix}\mathfrak b&\underline{\mathfrak n}^k\\0&\underline{a}\end{pmatrix}}{N(\lcm(\underline{\mathfrak n}^k))}+O(x^{d-1}\diam\lcm(\mathfrak b,\underline{\mathfrak n}^k))\right)= 
\\&=\frac{x^d}{N(\mathfrak b)}\displaystyle\prod_{\mathfrak p^k\nsupseteq \mathfrak b}\left(1-\frac{N(\gcd(\mathfrak p^k,\mathfrak b))  D(\mathfrak p^k,\mathfrak b\mid \underline{a})}{ N(\mathfrak p^k)  }\right)+ 
\\&+O\Bigg(x^d\sum_{i=1}^s\sum_{\substack{{N(\mathfrak n_i)>x^\alpha}\\{\underline{\mathfrak n}}}}  \frac1{N(\lcm(\mathfrak b,\underline{\mathfrak n}^k))}\Bigg)+O\Bigg(x^{d-1}  \sum_{N(\mathfrak n_i) \le x^\alpha}  \diam\lcm(\mathfrak b,\underline{\mathfrak n}^k)\Bigg).\label{eq:errors}
\end{align}
The first error term from \eqref{eq:errors} is at most a constant times
\begin{equation}x^d  \sum_{N(\mathfrak n)>x^\alpha}  \frac{d^s(\mathfrak n)}{N^k(\mathfrak n)}\ll_\eps \sum N^{-k+\eps}(\mathfrak n)\left(\frac{N(\mathfrak n)}{x^\alpha}\right)^{k-1-2\eps}\ll x^{d+\alpha(1-k+2\eps)}\zeta_K(1+\eps)
\ll x^{d-\alpha(k-1)+\eps}.\label{eq:error_term1}\end{equation}
The second error term is bounded by 
\begin{align}
\notag &\ll x^{d-1}\sum_{N(\mathfrak n_i)\le x^\alpha} N(\lcm(\mathfrak b,\underline{\mathfrak n}^k))\ll x^{d-1}\sum_{N(\mathfrak n)< x^{\alpha sk}} N(\mathfrak n)d^s(\mathfrak n)  \\
\notag &\ll x^{d-1}\sum N^{1+\eps}(\mathfrak n)\left(\frac{x^{\alpha sk}}{N(\mathfrak n)}\right)^{2+2\eps}\ll x^{d-1+\alpha sk(2+2\eps)}\zeta_K(1+\eps)\\
&\ll x^{d-1+2\alpha sk+\eps}.\label{eq:error_term2}
\end{align}
For $\Sigma_2$ we have 
$$|\Sigma_2|\le \sum_{j=1}^s f_j$$ and for each $j$
\begin{align*}|f_j|
&\le\sum_{\substack{\underline{\mathfrak n}\\{N(\mathfrak n_j)>x^\alpha}}}\sum_{\substack{{\| a\|\le x}\\{a+a_i\in\mathfrak n_i^k}\\{1\le i\le s}\\{a\in\mathfrak b}}} 1 
\ll \sum_{{N(\mathfrak n_j)>x^\alpha}}\sum_{\substack{{a+a_j\in \mathfrak n_j^k}\\{\| a\|\le x}}}\prod_{i\ne j}\sum_{a+a_i\in\mathfrak n_i^k} 1
\le \sum_{{N(\mathfrak n_j)>x^\alpha}} \sum_{\substack{ {a+a_j\in \mathfrak n_j^k}\\{\| a\|\le x}}}\prod_{i\ne j} d((a+a_i)).
\end{align*}
Since $d(\mathfrak n)\ll_\eps N^\eps(\mathfrak n)$, we can bound this by 
\[
\ll_\eps \sum_{{N(\mathfrak n_j)>x^\alpha}} \sum_{\substack{ {a+a_j\in \mathfrak n_j^k}\\{\| a\|\le x}}}\prod_{i\ne j} N^\eps((a+a_i)),\]
and from Lemma \ref{lem:added_lemma} we have
\begin{align}
\notag &\ll x^\eps \sum_{{N(\mathfrak n_j)>x^\alpha}}  \sum_{\substack{{a+a_j\in \mathfrak n_j^k}\\{\| a\|\le x}}}1
\ll x^\eps  \sum_{{N(\mathfrak n_j)>x^\alpha}} \sum_{\substack{{\mathfrak m\subseteq \mathfrak n^k}\\{N(\mathfrak m)\ll x^d}}}1\\
\notag &\ll \sum_{N(\mathfrak n)>x^\alpha}\left(\frac{x^d}{N(\mathfrak n^k)}\right)^{1+\eps}
\ll x^{d+d\eps}\sum_{N(\mathfrak n)>x^\alpha}\frac1{N^{k+k\eps}(\mathfrak n)}\left(\frac{N(\mathfrak n)}{x^\alpha}\right)^{k+k\eps-1-\eps}\\
 &\ll x^{d-\alpha(k-1)+\eps}.\label{eq:error_term3}
\end{align}
The three error terms come from equations \eqref{eq:error_term1}, \eqref{eq:error_term2}, and \eqref{eq:error_term3}; setting them equal gives $\alpha=\frac1{k+2sk-1}$ and proves the Proposition. 
\end{proof}
Proposition \ref{prop-Mirsky-like-formula} 
is a particular case of 
Theorem \ref{th:M=S+O} for $s=r+1$ and $\fb=\roi$.

%


\section{Construction of $\Lambda$ and the action $\Z^d\curvearrowright(\G,\mathrm{Haar})$}\label{Lambda-and-the-action-on-G}
In this section, using \emph{only the second correlation function}, we construct the groups $\Lambda$ and $\G$. 
Then we discuss an action $\Z^d\curvearrowright(\G,\mathrm{Haar})$ which has $\Lambda$ as spectrum.

For an ideal $\fa\subseteq\roi$, let us consider the annihilator $\fa^\perp$, i.e.\ 
the set of unitary characters $\chi\colon\roi\to\S^1$ such that $\chi(a)=1$ 
for all $a\in\fa$, see \cite{Schmidt-Dyn-Sys-Alg-Or}. 
Notice that $\#\fa^\perp
=\#\roi/\fa=N(\fa)$.
Throughout the paper, $\fd$ indicates a square-free ideal; equivalently, $\fd$ can be thought as a finite collection of prime ideals (or \emph{places}).

\begin{lemma}\label{lemma-spectral-measure}
Let us consider the measure 
\be
\nu=\sum_{\mu^2(\fd)=1}\sigma_{\fd}\sum_{\chi\in (\fd^k)^\perp}\delta_{\chi}\nonumber
\ee
on $\widehat{\roi}$, where 
\be\sigma_{\fd}= \displaystyle{\sum_{\substack{ \fb_0,\fb_1\subseteq\roi\\\mu^2(\fb_0)=\mu^2(\fb_1)=1\\\gcd(\fb_0,\fb_1)=\fd}  }\frac{\mu(\fb_0)\mu(\fb_1)}{\an{\lcm(\fb_0,\fb_1)^k}}}\label{def-sigma_d}.\ee
Then $\widehat\nu(a)=c_2(a)$, $a\in\roi$. 
\end{lemma}
We shall refer to $\nu$ as the \emph{spectral measure}.
Before proving this lemma, we need two additional results.
First, it is convenient to have another formula for $\sigma_\fd$ as an Euler product.
\begin{lemma}\label{lem-sigma_d=prod}
\be
\sigma_\fd=\frac{1}{N(\fd^k)}\prod_{\fp\nsupseteq\fd}\!\left(1-\frac{2}{N(\fp^k)}\right).\nonumber
\ee
\end{lemma}
\begin{proof} 
Multiply out the product in the RHS to get the sum (\ref{def-sigma_d}) defining $\sigma_\fd$.
\end{proof}
In particular, Lemma \ref{lem-sigma_d=prod} shows that $\sigma_\fd$ is positive and bounded away from zero and infinity. More precisely
\be
0<\prod_{\fp}\!\left(1-\frac{2}{N(\fp^2)}\right)=\sigma_{\roi}\leq\sigma_{\fd}<\frac{1}{\zeta_K(k)}\nonumber
\ee
and we can also write
\be
\sigma_{\fd}=\sigma_{\roi}\:\prod_{\fp\supseteq \fd}\frac{1}{N(\fp^2)-2}.\label{another-formula-for-sigma_d}
\ee

The second correlation function is the Fourier transform of a spectral measure whose atoms are weighted by the quantities $\sigma_\fd$. The following lemma allows us to write $c_2$ directly in terms of $\sigma_\fd$.
\begin{lemma}\label{lem-c2=sum} Let $a\in\roi$. Then 
\beq
c_2(a)=\sum_{\substack{ 
\fd^k\supseteq(a)\\\mu^2(\fd)=1  }}\sigma_{\fd},\label{lem-c_2(a)=sum-statement}
\eeq
and sum converges absolutely.
\end{lemma}
\begin{proof}
From Proposition \ref{prop-Mirsky-like-formula} we get
\be
D(\fp^k \mid 0,a)=\begin{cases}1&\mbox{if $\fp^k\supseteq (a)$;}\\2&\mbox{otherwise}.\end{cases}\nonumber
\ee
This gives
\be
c_2(a)=\prod_{\fp^k\supseteq (a)}\!\left(1-\frac{1}{N(\fp^k)}\right)\:\prod_{\fp^k\nsupseteq(a)}\!\left(1-\frac{2}{N(\fp^k)}\right).\label{lem-c_2(a)=sum-1}\ee
The sum in the RHS of (\ref{lem-c_2(a)=sum-statement}) converges absolutely by (\ref{another-formula-for-sigma_d}).  By Lemma \ref{lem-sigma_d=prod} we have
\begin{align}
\sum_{\substack{ \fd^k\supseteq(a)\\\mu^2(\fd)=1  }}\sigma_{\fd}&=
\sum_{\substack{ 
\fd^k\supseteq(a)\\ \mu^2(\fd)=1  }} \frac{1}{N(\fd^k)}\prod_{\fp\nsupseteq \fd}\!\left(1-\frac{2}{N(\fp^k)}\right)\nonumber
\\
&=\prod_{\fp}\!\left(1-\frac{2}{N(\fp^k)}\right)\sum_{\substack{ 
\fd^k\supseteq(a)\\\mu^2(\fd)=1  }}\frac{1}{N(\fd^k)}\prod_{\fp\supseteq\fd}\!\left(1-\frac{2}{N(\fp^k)}\right)^{-1}\nonumber
\\
&=\prod_{\fp}\!\left(1-\frac{2}{N(\fp^k)}\right)\sum_{\substack{ 
\fd^k\supseteq(a)\\\mu^2(\fd)=1  }}\prod_{\fp\supseteq\fd}\frac{1}{N(\fp^k)-2}\nonumber
\\
&=\prod_{\fp}\!\left(1-\frac{2}{N(\fp^k)}\right)\:\prod_{\fp^k\supseteq (a)}\!\left(1+\frac{1}{N(\fp^k)-2}\right)\nonumber
\\
&=\prod_{\fp^k\supseteq (a)}\!\left(1-\frac{1}{N(\fp^k)}\right)\:\prod_{\fp^k\nsupseteq(a)}\!\left(1-\frac{2}{N(\fp^k)}\right)=c_2(a),\nonumber
\end{align}
where the last equality comes from (\ref{lem-c_2(a)=sum-1}).
\end{proof}

\begin{proof}[Proof of Lemma  \ref{lemma-spectral-measure}]
Using Lemma \ref{lem-c2=sum} we can write
\be
c_2(a)=\sum_{\mu^2(\fd)=1} C_{\fd}(a),\mbox{ where } C_{\fd}(a)=\begin{cases}\sigma_{\fd},&\mbox{if $\fd^k
\supseteq (a)$;}\\0,&\mbox{otherwise}.\end{cases}\nonumber
\ee
The function $C_{\fd}$ is constant (equal to $\sigma_{\fd}$) on the lattice $\fd^k$ and zero elsewhere. This function on $\roi$ is the Fourier transform of a measure on $\widehat\roi$, given by a sum of Dirac $\delta$-measures at the points in the set $(\fd^k)^\perp$, with equal weights equal to $\sigma_{\fd}/N(\fd^k)$. The formula for the spectral measure $\nu$ and the lemma are proved.
\end{proof}

 Let $\Lambda$ be the support of the spectral measure $\nu$ defined above. It is automatically a group and,  by the Chinese Remainder Theorem for ideals,
\be
\Lambda=\bigcup_{\mu^2(\fd)=1}(\fd^k)^\perp 
\cong 
\bigoplus_{\fp}\roi/\fp^k.\label{Lambda=union}
\ee

Let us remark that the union in (\ref{Lambda=union}) is not disjoint. It will be useful for us to single out the smallest annihilator to which a character belongs. To this extent, let us notice that if $\fd_1\supseteq\fd_2$ then $
(\fd_1^k)^{\perp}\subseteq(\fd_2^k)^\perp$
and 
let us define the \emph{reduced annihilator} as
\be
(\fd^k)^\perp_{\textrm{red}}=\left\{\chi\in(\fd^k)^\perp\colon\chi\in(\fd'^k)^\perp \Rightarrow\fd\supseteq \fd'\right\}.
\ee
In other words 
\[ \chi\in(\fd^k)^\perp_{\mathrm{red}}\iff
\fd=\gcd\{\fd'\colon \mu^2(\fd')=1\mbox{ and }\chi\in(\fd'^k)^\perp\}.
\]

By Pontryagin duality (see, e.g., \cite{Hewitt-Ross}), $\widehat\Lambda$ is isomorphic to the compact abelian group
\be\G=\prod_\fp\roi/\fp^k.\label{def-G}\ee
Elements of $\G$ are coset sequences indexed by the set of prime ideals in $\roi$, i.e.\ ${\bf g}=
(g_{\fp^k}+\fp^k)_\fp$, where $g_{\fp^k}+\fp^k\in\roi/\fp^k$.  Given $\bf h\in\G$, we denote by $\T_{\bf h}$ the translation $\T_{\bf h}(\bf g)=\bf g+\bf h$. The Haar measure on $\G$ is the product of the counting measures on each factor $\roi/\fp^k$ and is defined on the natural Borel $\sigma$-algebra on $\G$. 

We have a $\Z^d$-action on $\Z^d\curvearrowright(\G,\mathrm{Haar})$ as follows: if $v\in\Z^d $ 
and ${\bf g}=(g_{\fp^k}+\fp^k)_{\fp}\in \G$, 
then 
\be v\cdot {\mathbf g}= \left(g_{\fp^k}+\iota(v)\right)_{\mathfrak p}\label{Z^d-action-on-G}.\ee
In other words, $\Z^d$ acts by $d$ commuting translations $\T_{{\bf u}_1},\ldots,\T_{{\bf u}_d}$ on $\G$, where
${\bf u}_i=(e_i+\fp^k)_{\fp}\in\G$.



Let us now discuss the spectrum of the action  (\ref{Z^d-action-on-G}).
For $v\in\Z^d$ let ${\bf U}_v$ be the unitary operator on  $\H=L^2(\G,\mathrm{Haar})$ given by
\be (\mathbf U_v f)(\mathbf g)=f(v\cdot \mathbf g).\nonumber\ee
\begin{prop}\label{prop-spectrum-Z^d-action}
The spectrum of $\Z^d\curvearrowright(\G,\mathrm{Haar})$ is isomorphic to $\Lambda$.
\end{prop}
\begin{proof}
Let $\iota\colon\Z^d\to\roi$ be the isomorphism 
defined in the Section \ref{sec2-mirsky-like}. Let $\mathbf g\in\G$ and for every prime ideal $\fp$ let $g_{\fp^k}+\fp^k\in\roi/\fp^k$ be its projection onto the $\fp^k$-th coordinate. Let $\chi\in(\fp^k)^\perp_{\mathrm{red}}$. Notice that if $a\equiv a'\bmod \fp^k$, then $\chi(a)=\chi(a')$; in other words, $\chi$ is well defined on $\roi/\fp^k$. Let $\xi_{\chi}(\mathbf g)=\chi(g_{\fp^k}+\fp^k)$. It is clear that $(\mathbf U_v\xi_\chi)(\mathbf g)=\chi(\iota(v))\xi_\chi(\mathbf g)$, i.e.\ $\xi$ is an eigenfunction with eigenvalue $\chi(\iota(v))$. If $\chi(\fd^k)^\perp_{\mathrm{red}}$ and $\fd=\fp_1\cdots\fp_s$ (distinct prime ideals), then $\chi=\chi_1\cdots\chi_s$, where $\chi_i\in(\fp_i^k)^\perp_{\mathrm{red}}$; in this case the function $\xi_{\chi}(\mathbf g)=\chi_1(g_{\fp_1^k}+\fp_1^k)\cdots\chi_{s}(g_{\fp_s^k}+\fp_s^k)$ is an eigenfunction for $\mathbf U_v$ with eigenvalue $\chi(\iota (v))$. Since characters are orthonormal with respect to the Haar measure on $\G$, we have that the discrete group $\{\chi\circ\iota\}_{\chi\in\Lambda}\subseteq\widehat{\Z^d}
=\T^d$ is the spectrum 
of the action $\Z^d\curvearrowright(\G,\mathrm{Haar})$ and is clearly isomorphic to $\Lambda$.
\end{proof}

\section{More Formul\ae\ for the Correlation Functions}\label{sec3}

The goal of this section is to prove three results that we will use later. The first one (Proposition \ref{lem:Hall}) is a generalization of a 
theorem by R.R. Hall \cite{Hall-1989}. 
Let $\eta\in\widehat\roi$ be the trivial character, $\eta(a)=1$ for every $a\in\roi$.
%

\begin{prop}\label{lem:Hall}
For every $r\geq1$ and every $a_1,\ldots,a_r\in\roi$ we have
\be
c_{r+1}(a_1\ldots,a_r)=\sum_{\fa_0}\sum_{\fa_1}\cdots\sum_{\fa_r}g(\fa_0)g(\fa_1)\cdots g(\fa_r)
\sum_{\substack{\chi_i\in(\fa_i^k)^\perp_{\mathrm{red}}\\ 0\leq i\leq r\\
\chi_0\chi_1\cdots\chi_r=\eta  }}\chi_0(0)\chi_1(a_1)\cdots\chi_r(a_r),\label{lemma-Hall-like-statement}
\ee
where
\be
g(\fa)=\frac{\mu(\fa)}{\zeta_K(k)}\prod_{\fp\supseteq\fa}\frac{1}{N(\fp^k)-1}.\label{def-g}
\ee
\end{prop}
The following lemmata follow from Proposition \ref{lem:Hall}, and deal with averages of the second and the third correlation functions, weighted by characters. These results are used in Section \ref{sec-spectrum-of-action} when studying the spectral properties of the action $\roi\curvearrowright(X,\Pi)$. Let us also point out that the proofs of Lemmata \ref{lem-g^2(d)} and \ref{lem:average-c_3} are considerably simpler than the proofs of the analogous results in \cite{Cellarosi-Sinai-JEMS}.

\begin{lemma}\label{lem-g^2(d)}
Let $\chi\in(\fd^k)^\perp_{\mathrm{red}}$. Then 
\be
\lim_{x\to\infty}\frac{1}{\#B_x}\sum_{b\in B_x}\chi(b) c_2(b)=g^2(\fd).\nonumber
\ee
\end{lemma}

\begin{lemma}\label{lem:average-c_3}
Let $\chi_1\in(\fd_1^k)^\perp_{\mathrm{red}}$, $\chi_2\in(\fd_2^k)^\perp_{\mathrm{red}}$, and $\chi=\chi_1\chi_2\in(\fd^k)^\perp_{\mathrm{red}}$. Then 
\be
\lim_{\substack{ x\to\infty\\y\to\infty  }}\frac{1}{\#B_x \#B_y}\sum_{b_1\in B_x}\sum_{b_2\in B_y}\chi_1(b_1)\chi_2(b_2) c_3(b_1,b_2)=g(\fd_1)g(\fd_2)g(\fd).\label{lem:average-c_3-statement}
\ee
\end{lemma}

Before discussing the proofs of the two lemmata above, let us give the

\begin{proof}[Proof of Proposition \ref{lem:Hall}]
Since $\chi_0(0)=1$ 
it will be omitted in proof of (\ref{lemma-Hall-like-statement}). We use notation $\underline{\mathfrak a}=(\mathfrak a_0,\dots,\mathfrak a_r)$ as in Section \ref{sec2-mirsky-like}.
Notice that the inner sum in (\ref{lemma-Hall-like-statement}) does not exceed
\be
\sum_{\substack{\chi_i\in(\fa_i^k)^\perp\\ 0\leq i\leq r\\ \chi_0\chi_1\cdots\chi_r=\eta}}1=\frac{N(\fa_0^k\fa_1^k\cdots \fa_r^k)}{N(\lcm(\underline{\fa}^k)}\nonumber
\ee
in absolute value. Moreover, for every ideal $\fa$, $|g(\fa)|\leq \frac{1}{N(\fa^k)}$ and the series
\be
\sum_{\fa_0}\sum_{\fa_1}\cdots\sum_{\fa_r}\frac{1}{N(\lcm(\underline{\fa}^k)}\nonumber
\ee
converges absolutely. Let us evaluate the inner sum in (\ref{lemma-Hall-like-statement}). Let $\fa=\lcm(\underline{\fa}^k)$. Notice that
\be
\frac{1}{N(\fa^k)}\sum_{a\in \roi/\fa^k}\chi_0(a)\chi_1(a)\cdots\chi_r(a)=\begin{cases}1&\mbox{if $\chi_0\chi_1\cdots\chi_r=\eta$;}\\ 0&\mbox{otherwise}.\end{cases}\nonumber
\ee
This allows us to rewrite the inner sum in (\ref{lemma-Hall-like-statement}) as
\begin{multline*}
\frac{1}{N(\fa^k)} \sum_{a\in\roi/\fa^k}\prod_{i=0}^r  \sum_{\chi_i\in(\fa_i^k)^\perp_{\mathrm{red}}}\chi_i(a_i+a)
=\frac{1}{N(\fa^k)}\sum_{a\in\roi/\fa^k}\prod_{i=0}^r\sum_{\substack{ \fb_i\supseteq \fa_i\\ \fb_i^k\supseteq (a_i+a)  }}\mu\!\left(\frac{\fa_i}{\fb_i}\right)N(\fb_i^k)\\
=\frac{1}{N(\fa^k)}\sum_{\fb_0\supseteq\fa_0}\sum_{\fb_1\supseteq \fa_1}\cdots\sum_{\fb_r\supseteq\fa_r}\mu\!\left(\frac{\fa_0}{\fb_0}\right)\mu\!\left(\frac{\fa_1}{\fb_1}\right)\cdots\mu\!\left(\frac{\fa_r}{\fb_r}\right)N^k(\underline{\fb})\sum_{\substack{  a\in\roi/\fa^k\\ a\equiv -a_i\bmod \fb_i^k\\ 0\leq i\leq r  }}1,
\end{multline*}
where $a_0=0$ and $\frac{\fa_j}{\fb_j}$ denotes the unique ideal $\fc_j$ such that $\fa_j=\fb_j \fc_j$.
Observe that 
\be
\sum_{\substack{  a\in\roi/\fa^k\\ a\equiv -a_i \bmod \fb_i^b\\ 0\leq i\leq r  }}1=E\begin{pmatrix}\fb_0^k&\underline{\fb}^k\\ 0&\underline{a}\end{pmatrix}\frac{N(\fa^k)}{N(\lcm(\underline{\fb}^k))}\nonumber
\ee
and thus the inner sum in (\ref{lemma-Hall-like-statement})  equals
\be
\sum_{\fb_0\supseteq\fa_0}\sum_{\fb_1\supseteq \fa_1}\cdots\sum_{\fb_r\supseteq\fa_r}\mu\!\left(\frac{\fa_0}{\fb_0}\right)\mu\!\left(\frac{\fa_1}{\fb_1}\right)\cdots\mu\!\left(\frac{\fa_r}{\fb_r}\right)
\frac{N^k(\underline{\fb})} {N(\lcm(\underline{\fb}^k))}E\begin{pmatrix}\fb_0^k&\underline{\fb}^k\\ 0&\underline{a}\end{pmatrix}\label{proof-hall-1}.
\ee
Notice that the $\fb_i$'s 
are necessarily square-free and thus $\mu(\fa_i/\fb_i)
=\mu(\fa_i)\mu(\fb_i)$.
Let us also observe that, 
for $i=0,\ldots,r$,
\be
\sum_{\substack{ \fa_i\subseteq \fb_i  }}\mu(\fa_i)g(\fa_i)=\frac{\mu^2(\fb_i)}{N(\fb_i^k)}.\label{we-observe}
\ee
To see this, for $\mu^2(\fb_i)=1$, we can write the LHS of (\ref{we-observe}) as
\begin{align*}
&\frac{1}{\zeta_K(k)}\sum_{\fa_i\subseteq \fb_i}\prod_{\fp\supseteq \fa_i}\left(\frac{1}{N(\fp^k)}+\frac{1}{N(\fp^{2k})}+\frac{1}{N(\fp^{3k})}+\ldots\right)\\
&=\frac{1}{\zeta_K(k)}\sum_{\fa_i\subseteq \fb_i}\prod_{\fp\supseteq \fb_i}\left(\frac{1}{N(\fp^k)}+\frac{1}{N(\fp^{2k})}+\ldots\right)\prod_{\substack{ \fp\supseteq \fa_i\\\fp\nsupseteq \fb_i  }}\left(\frac{1}{N(\fp^k)}+\frac{1}{N(\fp^{2k})}+\ldots\right)\\
&=\frac{1}{\zeta_K(k)}\frac{1}{N(\fb_i^k)}\prod_{\fp\supseteq \fb_i}\left(1+\frac{1}{N(\fp^k)}+\frac{1}{N(\fp^{2k})}+\ldots\right)\sum_{\fa_i\subseteq \fb_i}\prod_{\substack{ \fp\supseteq \fa_i\\\fp\nsupseteq \fb_i  }}\frac{1}{N(\fp^k)-1}\\
&=\frac{1}{\zeta_K(k)}\frac{1}{N(\fb_i^k)}\prod_{\fp\supseteq \fb_i}\!\left(1-\frac{1}{N(\fp^k)}\right)^{-1}\prod_{\fp\nsupseteq \fb_i}\!\left(1+\frac{1}{N(\fp^k)-1}\right)=\frac{1}{N(\fb_i^k)},\nonumber
\end{align*}
since the products combined give the Euler product for $\zeta_K(k)$. Alternatively, one can expand the products into sums and match terms with the series defining $\zeta_K(k)$. Now (\ref{proof-hall-1}) and (\ref{we-observe}) imply that  the multiple sum in (\ref{lemma-Hall-like-statement}) equals
\be
\sum_{\fb_0}\sum_{\fb_1}\cdots\sum_{\fb_r}\frac{\mu(\fb_0)\mu(\fb_1)\cdots\mu(\fb_r)}{N(\lcm(\underline{\fb}^k))}E\begin{pmatrix}\fb_0^k&\underline{\fb}^k\\ 0&\underline{a}\end{pmatrix},\nonumber
\ee
and by Lemma \ref{lem:mirsky6} we get the desired statement.
\end{proof}

%

\begin{proof}[Proof of Lemma \ref{lem-g^2(d)}]
Observe that 
\be\lim_{x\to\infty}\frac{1}{\# B_x}\sum_{b\in B_x}\chi(b)\chi_1(b)=\begin{cases}1&\mbox{if $\chi_1=\chi^{-1}$;}\\0&\mbox{otherwise.}\end{cases}\label{obs-average-characters}\ee
By Proposition \ref{lem:Hall}
\begin{align*}
\lim_{x\to\infty}\frac{1}{\#B_x}\sum_{b\in B_x}\chi(b)c_2(b)&= \lim_{x\to\infty}\frac{1}{\#B_x}\sum_{b\in B_x}\chi(b)\sum_{\fa_0,\fa_1}g(\fa_0)g(\fa_1)\!\!\!\!\sum_{\substack{ \chi_i\in(\fa_i^k)^\perp_{\mathrm{red}}\\i=0,1\\\chi_0\chi_1=\eta  }}
\chi_1(b)\\
&=\sum_{\fa_0,\fa_1}g(\fa_0)g(\fa_1)\sum_{\substack{ \chi_i\in(\fa_i^k)^\perp_{\mathrm{red}}\\i=0,1\\\chi_0\chi_1=\eta  }}\lim_{x\to\infty}\frac{1}{\#B_x}\sum_{b\in B_x}\chi(b)\chi_1(b)\\
&=g(\fd)\sum_{\fa_0}g(\fa_0)\sum_{\substack{ \chi_0\in(\fa_0^k)^\perp_{\mathrm{red}}\\\chi_0\chi^{-1}=\eta  }}1=g^2(\fd).
\end{align*}
\end{proof}

\begin{proof}[Proof of Lemma \ref{lem:average-c_3}]
Using (\ref{obs-average-characters}) and Proposition \ref{lem:Hall} the LHS of (\ref{lem:average-c_3-statement}) can be written as
\begin{multline*}
\sum_{\fa_0,\fa_1,\fa_2}g(\fa_0)g(\fa_1)g(\fa_2)
\!\!\!\!\!\sum_{\substack{ \chi_i'\in(\fa_i^k)^\perp_{\mathrm{red}}\\i=0,1,2\\\chi_0'\chi_1'\chi_2'=\eta  }}\lim_{\substack{ x\to\infty\\y\to\infty  }}\frac{1}{\#B_x\#B_y}\sum_{b_1\in B_x}\sum_{b_2\in B_y}\chi_1(b_1)\chi_2(b_2)\chi_1'(b_1)\chi_2'(b_2)\\
=\sum_{\fa_0}g(\fa_0)g(\fd_1)g(\fd_2)\!\!\!\!\!\sum_{\substack{ \chi_0\in(\fa_0^k)^\perp_{\mathrm{red}}\\\chi_0'\chi_1^{-1}\chi_2^{-1}=\eta  }}1=g(\fd_1)g(\fd_2)g(\fd).
\end{multline*}
\end{proof}

\section{The Action $\roi\curvearrowright(X,\Pi)$}\label{sec-spectrum-of-action}

Consider the space $X=\{0,1\}^{\roi}$, whose elements are $\roi$-indexed sequences $x=(x(a))_{a\in\roi}$, equipped with the 
Borel $\sigma$-algebra generated by cylinder sets. Introduce on $X$ the probability measure $\Pi$ 
defined as follows: for every $r\geq0$ and every $a_0, a_1,\ldots,a_r\in\roi$,
\beq
\Pi\left\{x\in X\colon  x(a_0)=x(a_1)=\dots=x(a_r)=1\right\}=\zeta_K(k) c_{r+1}(a_1-a_0, a_2-a_0,\ldots, a_r-a_0),\label{def-Pi-on-cylinders}
\eeq
where $c_{r+1}$ is the $(r+1)$-st correlation function (\ref{general-c_r+1-Mirsky}) associated to $\mathcal F_k$. It is clear that (\ref{def-Pi-on-cylinders}) determines the measure $\Pi$ uniquely. We call $\Pi$ \emph{the natural measure corresponding to the set of $k$-free integers in $\roi$}. 

If we consider the $\roi$-action on $X$ defined as $b\cdot x=(x(a+b))_{a\in\roi}$ for $b\in\roi$ and $x\in X$, then it follows immediately from (\ref{def-Pi-on-cylinders}) that $\Pi$ is invariant under this action. 
We can now reformulate the main result of this paper, of which Theorem \ref{thm-1-intro} is a simplified version.
\begin{theorem}[Main Theorem, second version]\label{main-thm-second-version}
\begin{itemize}
\item[(i)] The action $\roi\curvearrowright(X,\Pi)$  is ergodic and has pure point spectrum given by $\Lambda$.
\item[(ii)] 
The two actions
$\roi\curvearrowright(X, \Pi)$ and $\Z^d\curvearrowright(\G,\mathrm{Haar})$ given in (\ref{Z^d-action-on-G}) are isomorphic.
\end{itemize}
\end{theorem}


For $a\in\roi$, let $U_a$ be the unitary operator on $\mathcal H=L^2(X,\Pi)$ given by
\be
(U_a f)(x)=f(a\cdot x).\nonumber
\ee

The proof of Theorem \ref{main-thm-second-version}-(i) 
 requires us to show that there exists an orthonormal basis $\{\theta_\chi\}_{\chi\in\Lambda}$ for $L^2(X,\Pi)$ such that 
$U_{a}\theta_\chi=\chi(a)\theta_\chi$.
First we will show that $\Lambda$ is contained in the spectrum of the action $\roi\curvearrowright(X,\Pi)$. 
For $\chi\in\Lambda$, let us define the function $\theta_\chi\colon X\to\C$,
\be
\theta_\chi(x):=\lim_{R\to\infty}\frac{1}{\#B_R}\sum_{a\in B_R}\chi(-a)x(a)\label{def-theta_chi}
\ee
\begin{prop}\label{theorem-existence-eigenfunctions}
Let $\chi\in\Lambda$. Then \eqref{def-theta_chi}  defines a function $\theta_\chi\in\mathcal H$, satisfying 
\be
(U_a\theta_\chi)(x)=\chi(a)\theta_\chi(x)\label{theta_chi-eigenfunction-of-U_a}
\ee
for $\Pi$-almost every $x\in X$.
\end{prop}
\begin{proof}
Let $f_0\in\mathcal H$, $f_0(x)=x(0)$, and for $a\in\roi$ let $U_{a,\chi}$ be the unitary operator on $\mathcal H$ defined by
\be(U_{a,\chi}f)(x)=\chi(-a)f(a\cdot x).\nonumber\ee
Since $\roi$ is amenable, von Neumann's mean ergodic theorem for $\roi$-actions holds and implies that the limit
\be
\lim_{R\to\infty}\frac{1}{\#B_R} \sum_{a\in B_R} U_{a,\chi} f_0 \nonumber
\ee
exists in $\mathcal H$. 
For $\Pi$-almost every $x\in X$, we have 
\begin{align}
\lim_{R\to\infty}\frac{1}{\#B_R} \sum_{a\in B_R} U_{ a,\chi} f_0(x)
&=
\lim_{R\to\infty}\frac{1}{\#B_R}\sum_{a\in B_R}\chi(-a)f_0(a\cdot x)\notag
\\&=
\lim_{R\to\infty}\frac{1}{\#B_R}\sum_{a\in B_R}\chi(-a)x(a)=\theta_\chi(x).
\end{align}
Since $\theta_\chi$ is $U_{a,\chi}$-invariant, i.e.\ $(U_{a,\lambda}\theta_\chi)(x)=\theta_\chi(x)$ for $\Pi$-almost every $x\in X$, we get that $\chi(-a)\theta_\chi(a\cdot x)=\theta_\chi(x)$, i.e.
(\ref{theta_chi-eigenfunction-of-U_a}).
\end{proof}
For $a\in\roi$ let us denote by $x(a)$ the function $X\to\{0,1\}$ given by the projection of $x\in X$ onto its $a$-th coordinate. 
We have the 
\begin{prop}\label{eigenfunctions-are-nonzero}
The functions $\theta_\chi$ defined in (\ref{def-theta_chi}) are nonzero.
\end{prop}
\begin{proof}
It is enough to show that the inner products $\langle x(a),\theta_\chi\rangle$, for $a\in\roi$, are  in general nonzero. We actually prove something more, that is an explicit formula for these inner products. 
Let $\chi\in(\fd^k)^\perp_{\mathrm{red}}$. We claim that for every $a\in\roi$ we have
\be
\langle x(a),\theta_\chi\rangle=\zeta_K(k)\chi(a)g^2(\fd),\label{statement-prop-inner-product-x(a)-and-theta_chi}
\ee
where $g$ is the function defined in Proposition \ref{lem:Hall}.
To see this, observe that $\langle x(a),x(b)\rangle=\zeta_K(k) c_2(b-a)$. From (\ref{def-theta_chi}) and Lemma \ref{lem-g^2(d)} we get
\begin{align*}
\langle x(a),\theta_\chi\rangle
&=\lim_{R\to\infty}\left\langle x(a),\frac{1}{|B_R|}\sum_{b\in B_R}\chi(-b)x(b)\right\rangle 
\lim_{R\to\infty}\frac{1}{\#B_R}\sum_{b\in B_R}\chi(b)\langle x(a),x(b)\rangle
\\&=
\lim_{R\to\infty}\frac{1}{\#B_R}\sum_{b\in B_R}\chi(b)\zeta_K(k) c_2(b-a)
\zeta_K(k)\chi(a)\lim_{R\to\infty}\frac{1}{\#B_R}\sum_{b\in B_R}\chi(b)c_2(b)
\\&=
\zeta_K(k)\chi(a)g^2(\fd).
\end{align*}
\end{proof}

Propositions \ref{theorem-existence-eigenfunctions} and \ref{eigenfunctions-are-nonzero} show that $\Lambda$ is contained in the spectrum of the action $\roi\curvearrowright(X,\Pi)$. 
Notice that, since $U_a$ is a unitary operator for every $a\in\roi$, the eigenfunctions $\theta_\chi$ are orthogonal to one another for different $\chi\in\Lambda$.
Introduce the distinguished subspace $H\subseteq \mathcal H$,
\be
H=\overline{\left\{\sum_{a}z_a x(a)\right\}},
\ee
i.e.\ the closure of the set of all complex linear combinations of the $x(a)$'s. Notice that $H$ is $U_a$-invariant for every $a\in \roi$ and, by (\ref{def-theta_chi}), all the 
functions $\theta_\chi$ belong to $H$.
Let us write
\be
x(a)=\sum_{\chi\in\Lambda}\langle x(a),\theta_\chi\rangle\theta_\chi.\nonumber\ee

An important step in the proof of the Main Theorem 
is given by the 
\begin{prop}\label{prop-theta_chi-basis-for-H}
The family of eigenfunctions $\{\theta_\chi\}_{\chi\in\Lambda}$ is a basis for $H$.
\end{prop}
\begin{proof}
By orthogonality, it is enough to show that the eigenfunctions span the space of all linear combinations of the $x(a)$'s. 
Let us show that $H$ is isomorphic to $L^2(\widehat{\roi},\nu)$, where $\nu$ is the spectral measure.

The function $x\mapsto x(0)$ belongs to $L^2(X,\Pi)$ and for every $a\in\roi$, we have $\langle U_ax(0),x(0)\rangle=c_2(a)$. 
 Notice that 
\be
H=\overline{\mathrm{span}\left\{U_ax(0)\colon a\in\roi\right\}}.\nonumber
\ee
By definition of $\nu$, $\hat\nu(a)=c_2(a)$, that is
\be
\int_{\widehat\roi}i(a)(\chi)\mathrm{d}\nu(\chi)=\langle U_ax(0),x(0)\rangle,\nonumber
\ee
where $i\colon\roi\to\widehat{\widehat{\roi}}$ is the canonical isomorphism, $i(a)(\chi)=\chi(a)$. The map $L^2(X,\Pi)\to L^2(\widehat\roi,\nu)$, $U_a x(0)\mapsto i(a)$ is an equivariant isometry (with respect to the $\roi$-action). This extends to a unitary operator $W\colon H\to L^2(\widehat\roi,\nu)$ and yields an isomorphism between the unitary representation $U|_H$ and $V^{\nu}$, where
$
a\mapsto (U|_H)_a:=U_a|_H\colon H\to H,
(U_a|_Hf)(x)=f(a\cdot x)$
and
$
a\mapsto V^\nu\colon L^2(\widehat\roi,\nu)\to L^2(\widehat\roi,\nu),
(V^\nu_a f)(\chi)=i(a)(\chi)f(\chi)=\chi(a)f(\chi).$
In particular, we have
\be
H\cong L^2(\widehat\roi,\nu)=\bigoplus_{\chi\in\Lambda}L^2(\widehat\roi,\sigma_{\chi}\delta_\chi),\label{H-isom-L2(hat-roi-nu)}
\ee
where $$\sigma_\chi=\sum_{\substack{ \mu^2(\fd)=1,  
\chi\in(\fd^k)^\perp  }}\sigma_\fd.$$ Since we have constructed a non-trivial eigenfunction $\theta_\chi$ for each $\chi\in\Lambda$, (\ref{H-isom-L2(hat-roi-nu)}) implies that the family of eigenfunctions $\{\theta_\chi\}_{\chi\in\Lambda}$ spans $H$.
\end{proof}

We want to normalize each eigenfunction to make the family $\{\theta_\chi\}_{\chi\in\Lambda}$ an orthonormal basis for $H$. 
The function $g$ defined in \eqref{def-g} plays again an essential role:
\begin{lemma}\label{lemma-L2-norm-or-theta_chi}
Let $\chi\in(\fd^k)^\perp_{\mathrm{red}}$. Then 
\be
\|\theta_\chi\|=\sqrt{\zeta_K(k)}|g(\fd)|.\nonumber
\ee
\end{lemma}
\begin{proof}
From (\ref{statement-prop-inner-product-x(a)-and-theta_chi}) 
we get
\begin{align*}
\|\theta_\chi\|^2
&=
\langle \theta_\chi,\theta_\chi\rangle=\left\langle\theta_\chi,\lim_{R\to\infty}\frac{1}{\#B_R}\sum_{a\in B_R}\chi(-a)x(a)\right\rangle=
\lim_{R\to\infty}\frac{1}{\#B_R}\sum_{a\in B_R}\chi(a)\overline{\langle x(a),\theta_\chi\rangle}
\\&=
\lim_{R\to\infty}\frac{1}{\#B_R}\sum_{a\in B_R}|\chi(a)|^2\zeta_K(k)g^2(\fd)=\zeta_K(k)g^2(\fd).
\end{align*}
\end{proof}
Let us denote the the normalized eigenfunctions by
$\tilde\theta_\chi=\theta_\chi/\|\theta_\chi\|$. 
If we write $x(a)=\sum_{\chi\in\Lambda}\langle x(a),\tilde\theta_\chi\rangle\tilde\theta_\chi$, then we can retrieve the fact that $\|x(a)\|=1$ using Proposition \ref{prop-theta_chi-basis-for-H} and Lemma \ref{lemma-L2-norm-or-theta_chi}:
\begin{align*}
\|x(a)\|&=\sum_{\chi\in\Lambda}\left|\langle x(a),\tilde\theta_\chi\rangle\right|^2=\sum_{\mu^2(\fd)=1}\sum_{\chi\in (\fd^k)^\perp_{\mathrm{red}}}\zeta_K(k)g^2(\fd)
\\&=
\zeta_K(k)\sum_{\mu^2(\fd)=1}g^2(\fd)\#(\fd^k)^\perp_{\mathrm{red}}
\\&=
\frac{1}{\zeta_K(k)}\sum_{\mu^2(\fd)=1}\prod_{\fp\supseteq \fd}\frac{1}{N(\fp^k)-1}=\frac{1}{\zeta_K(k)}\prod_{\fp}\!\left(1+\frac{1}{N(\fp^k)-1}\right)
\\&=
\frac{1}{\zeta_K(k)}\prod_{\fp}\left(1-\frac{1}{N(\fp^k)}\right)^{-1}=1.\nonumber
\end{align*}
The same argument allows us to provide an approximation of the function $x(a)$ for $a\in\roi$: let $D\geq 1$ and define
\be
x_D(a)=\sum_{\substack{ \mu^2(\fd)=1\\ N(\fd)\leq D  }}\sum_{\chi\in (\fd^k)^\perp_{\mathrm{red}}}\left\langle x(a),\tilde\theta_{\chi}\right\rangle \tilde\theta_\chi.\nonumber
\ee
We have the following estimate 
\beq
\|x(a)-x_D(a)\|^2 =
\sum_{\substack{ \mu^2(\fd)=1\\ N(\fd)> D  }}\sum_{\chi\in (\fd^k)^\perp_{\mathrm{red}}}\left|\left\langle x(a),\tilde\theta_{\chi}\right\rangle\right|^2=\sum{N(\fd)>D}|g(\fd)|= O(D^{-1+\varepsilon})\label{approx-x(a)-by-x_D(a)}
\eeq
for every $\varepsilon>0$.
Another important step in the proof of the Main  Theorem 
is to show that the pointwise product of two eigenfunctions is still an eigenfunction. This is a peculiarity of actions with pure-point spectrum. 
\begin{prop}\label{prop-product-eigenfunctions}
Let $\chi_1\in(\fd_1^k)^\perp_{\mathrm{red}}$ and $\chi_2\in(\fd_2^k)^\perp_{\mathrm{red}}$. Then
\be
\tilde\theta_{\chi_1}\tilde\theta_{\chi_2}=\epsilon\tilde\theta_{\chi},\nonumber
\ee
where $\chi=\chi_1\chi_2\in(\fd^k)^\perp_{\mathrm{red}}$ and $\epsilon=\mu(\fd_1)\mu(\fd_2)\mu(\fd)$.
\end{prop}
\begin{proof}
It is enough to show that for every $a\in\roi$ we have
\be
\left\langle\tilde\theta_{\chi_1}\tilde\theta_{\chi_2},x(a)\right\rangle=\epsilon\left\langle\tilde\theta_\chi,x(a)\right\rangle.\nonumber
\ee
Using the definition (\ref{def-theta_chi}) we have
\[
\theta_{\chi_1}\theta_{\chi_2}=\lim_{\substack{ R_1\to\infty\\ R_2\to\infty  }}\frac{1}{\#B_{R_1}\#B_{R_2}}\sum_{a_1\in B_{R_1}}\sum_{a_2\in B_{R_2}}\chi_1(-a_1)\chi_2(-a_2)x(a_1)x(a_2)\nonumber
\]
and thus
\begin{align*}
\left\langle\theta_{\chi_1}\theta_{\chi_2},x(a)\right\rangle
&=\lim_{\substack{ R_1\to\infty\\ R_2\to\infty  }}\frac{1}{\#B_{R_1}\#B_{R_2}}\sum_{a_1\in B_{R_1}}\sum_{a_2\in B_{R_2}}\chi_1(-a_1)\chi_2(-a_2)\left\langle x(a_1)x(a_2),x(a)\right\rangle=
\\&=
\zeta_K(k)\lim_{\substack{ R_1\to\infty\\ R_2\to\infty  }}\frac{1}{\#B_{R_1}\#B_{R_2}}\sum_{a_1\in B_{R_1}}\sum_{a_2\in B_{R_2}}\chi_1(-a_1)\chi_2(-a_2)c_3(a_1-a,a_2-a)
\\&=
\zeta_K(k)(\chi_1\chi_2)(-a)\lim_{\substack{ R_1\to\infty\\ R_2\to\infty  }}\frac{1}{\#B_{R_1}\#B_{R_2}}\sum_{a_1\in B_{R_1}}\sum_{a_2\in B_{R_2}}\chi_1(-a_1)\chi_2(-a_2)c_3(a_1,a_2)
\\&=
\zeta_K(k)\chi(-a)g(\fd_1)g(\fd_2)g(\fd)\nonumber
\end{align*}
by Lemma \ref{lem:average-c_3}. On the other hand, by (\ref{statement-prop-inner-product-x(a)-and-theta_chi}),
\be
\left\langle \theta_\chi,x(a)\right\rangle=\zeta_K(k)\chi(-a)g^2(\fd).\nonumber
\ee
Therefore
\be
\epsilon=\left\langle\tilde\theta_{\chi_1}\tilde\theta_{\chi_2},x(a)\right\rangle\left\langle\tilde\theta_{\chi},x(a)\right\rangle^{-1}=\frac{g(\fd_1)g(\fd_2)g(\fd)}{|g(\fd_1)||g(\fd_2)|}\frac{|g(\fd)|}{g^2(\fd)}=\mu(\fd_1)\mu(\fd_2)\mu(\fd).\nonumber
\ee
\end{proof}
So far, we have proven that the family of eigenfunctions $\{\tilde\theta_\chi\}_{\chi\in\Lambda}$ is an orthonormal family for the subspace $H$, and that these eigenfunction have a remarkable multiplicative property. Now we want to show that the subspace $H$ coincides with the full Hilbert space $\mathcal H$. This will imply that $\{\tilde\theta_\chi\}_{\chi\in\Lambda}$ is in fact an orthonormal basis for $\mathcal H$ and therefore there is no ``room'' for other eigenspaces. In other words, $\Lambda$ gives \emph{all} the spectrum.
\begin{prop}\label{prop-H=mathcalH}
$H=\mathcal H$.
\end{prop}
\begin{proof}
Let $\mathcal B$ denote the Borel $\sigma$-algebra on $X$. 
The space $(X,\mathcal B, \Pi)$ is Lebesge in the sence of Rokhlin \cite{Rohlin-1947}.
By Proposition \ref{prop-product-eigenfunctions}, the space $H$ is a sub-ring of the unitary ring $\mathcal H=L^2(X,\mathcal B,\Pi)$. Rokhlin's theorem \cite{Rohlin-1948} implies that $H=L^2(X,\mathcal F,\Pi|_{\mathcal F})$, where $\mathcal F$ is a $\sigma$-subalgebra of $\mathcal B$. We claim that, up to null sets, $\mathcal F=\mathcal B$. 
Let us assume for contradiction that $\mathcal F\subsetneq\mathcal B$, i.e.\ there is a positive measure set in $\mathcal B\smallsetminus\mathcal F$. The conditional expectation operator $\E(\cdot|\mathcal F)$ is an orthogonal projection $\mathcal H\to H$, that is $\E(f|\mathcal F)=\mathrm{proj}_H(f)$ for every $f\in\mathcal H$. There exist a function $f\in\mathcal H$ and a constant $\alpha>0$ such that $\|f-\E(f|\mathcal F)\|\geq\alpha$. Let $\varepsilon>0$, and let $f'=\sum_{i=1}^n \alpha_i {\bf 1}_{A_i}$ be a simple function such that $\|f-f'\|\leq\frac{\varepsilon}{2}$, where $A_i$ are cylinders (that is $A_i=\{x\in X\colon  x(a_1^{(i)})x(a_2^{(i)})\cdots x(a_{r^{(i)}}^{(i)})=1\}$ for some $r^{(i)}\geq1$ and  $a_1^{(i)},a_2^{(i)},\ldots,a_{r^{(i)}}^{(i)}\in\roi$).
By (\ref{approx-x(a)-by-x_D(a)}) each function $x(a_j^{(i)})$ can be approximated arbitrarily well by a linear combination of the $\tilde\theta_{\chi}$'s and thus there exists a polynomial in the $\tilde\theta_\chi$'s, say $f''$, such that $\|f'-f''\|\leq \frac{\varepsilon}{2}$. By Proposition \ref{prop-product-eigenfunctions} the function $f''$ can be written as a \emph{linear} combination of the $\tilde\theta_{\chi}$'s, and therefore, by Proposition \ref{prop-theta_chi-basis-for-H}, $f''\in H$.
This means that we are able to find $f\in H$ such that $\|f-f''\|\leq \varepsilon$ and, if $\varepsilon$ is sufficiently small, this contradicts the fact that $\|f-\mathrm{proj}_H(f)\|\geq\alpha$.
\end{proof}
Propositions \ref{prop-theta_chi-basis-for-H} and \ref{prop-H=mathcalH} immediately give  
\begin{corollary}
The family of eigenfunctions $\{\tilde\theta_\chi\}_{\chi\in\Lambda}$ is an orthonormal basis for $\mathcal H$.
\end{corollary}
This fact, together with Propositions \ref{theorem-existence-eigenfunctions} and 
\ref{eigenfunctions-are-nonzero}, yields 
part (i) of Theorem \ref{main-thm-second-version}. Theorem \ref{thm-1-intro} (i) follows immediatley since uniqueness in \eqref{def-Pi-on-cylinders} is guaranteed by Kolmogorov consistency.
%
%
%
%
Finally, Proposition \ref{prop-spectrum-Z^d-action} and
Mackey's theorem \cite{Mackey-1964} imply that 
the two actions $\roi\curvearrowright(X,\Pi)$ and $\Z^d\curvearrowright(\G,\mathrm{Haar})$ are isomorphic.
This constitutes part (ii) of Theorem \ref{main-thm-second-version}, which gives Theorem \ref{thm-1-intro} (ii). 
%

\bibliographystyle{plain}
\bibliography{square-free-bibliography}

\begin{thebibliography}{10}

\bibitem{Baake-Moody-Pleasants-2000}
M.~Baake, R.V. Moody, and P.A.B. Pleasants.
\newblock Diffraction from visible lattice points and {$k$}th power free
  integers.
\newblock {\em Discrete Math.}, 221(1-3):3--42, 2000.
\newblock Selected papers in honor of Ludwig Danzer.

\bibitem{Bergelson-Gorodnik-2004}
V.~Bergelson and A.~Gorodnik.
\newblock Weakly mixing group actions: a brief survey and an example.
\newblock In {\em Modern dynamical systems and applications}, pages 3--25.
  Cambridge Univ. Press, Cambridge, 2004.

\bibitem{Cellarosi-Sinai-JEMS}
F~Cellarosi and Ya.G. Sinai.
\newblock Ergodic properties of square-free numbers.
\newblock {\em J. Eur. Math. Soc. (JEMS)}, 15(4):1343--1374, 2013.

\bibitem{Hall-1989}
R.R. Hall.
\newblock The distribution of squarefree numbers.
\newblock {\em J. Reine Angew. Math.}, 394:107--117, 1989.

\bibitem{Halmos-vonNeumann-1942}
P.R. Halmos and J.~von Neumann.
\newblock Operator methods in classical mechanics. {II}.
\newblock {\em Ann. of Math. (2)}, 43:332--350, 1942.

\bibitem{Heath-Brown-1984}
D.R. Heath-Brown.
\newblock The square sieve and consecutive square-free numbers.
\newblock {\em Math. Ann.}, 266(3):251--259, 1984.

\bibitem{Hewitt-Ross}
E.~Hewitt and K.A. Ross.
\newblock {\em Abstract harmonic analysis. {V}ol. {I}}, volume 115 of {\em
  Grundlehren der Mathematischen Wissenschaften [Fundamental Principles of
  Mathematical Sciences]}.
\newblock Springer-Verlag, Berlin, second edition, 1979.
\newblock Structure of topological groups, integration theory, group
  representations.

\bibitem{Karatzas_Shreve}
Ioannis Karatzas and Steven~E. Shreve.
\newblock {\em Brownian motion and stochastic calculus}, volume 113 of {\em
  Graduate Texts in Mathematics}.
\newblock Springer-Verlag, New York, second edition, 1991.

\bibitem{Koralov_Sinai}
Leonid~B. Koralov and Yakov~G. Sinai.
\newblock {\em Theory of probability and random processes}.
\newblock Universitext. Springer, Berlin, second edition, 2007.

\bibitem{Sarnak-Liu}
J.~Liu and P.~Sarnak.
\newblock The {M}\"{o}bius function and distal flows.
\newblock {\em preprint}.

\bibitem{Mackey-1964}
G.~W. Mackey.
\newblock Ergodic transformation groups with a pure point spectrum.
\newblock {\em Illinois J. Math.}, 8:593--600, 1964.

\bibitem{Mirsky-1949}
L.~Mirsky.
\newblock Arithmetical pattern problems relating to divisibility by {$r$}th
  powers.
\newblock {\em Proc. London Math. Soc. (2)}, 50:497--508, 1949.

\bibitem{Narkiewicz}
W.~Narkiewicz.
\newblock {\em Elementary and analytic theory of algebraic numbers}.
\newblock Springer Monographs in Mathematics. Springer-Verlag, Berlin, third
  edition, 2004.

\bibitem{Neukirch}
J.~Neukirch.
\newblock {\em Algebraic number theory}, volume 322 of {\em Grundlehren der
  Mathematischen Wissenschaften [Fundamental Principles of Mathematical
  Sciences]}.
\newblock Springer-Verlag, Berlin, 1999.

\bibitem{Peckner-2012}
R.~Peckner.
\newblock Uniqueness of the measure of maximal entropy for the squarefree flow.
\newblock {\em preprint}, 2012.

\bibitem{Pleasants-Huck}
P.A.B. Pleasants and C.~Huck.
\newblock Entropy and {D}iffraction of the {$k$}-{F}ree {P}oints in
  {$n$}-{D}imensional {L}attices.
\newblock {\em Discrete Comput. Geom.}, 50(1):39--68, 2013.

\bibitem{Rohlin-1947}
V.A. Rokhlin.
\newblock On the problem of the classification of automorphisms of {L}ebesgue
  spaces.
\newblock {\em Doklady Akad. Nauk SSSR (N. S.)}, 58:189--191, 1947.

\bibitem{Rohlin-1948}
V.A. Rokhlin.
\newblock Unitary rings.
\newblock {\em Doklady Akad. Nauk SSSR (N.S.)}, 59:643--646, 1948.

\bibitem{Sarnak-Mobius-lectures}
P.~Sarnak.
\newblock Three lectures on the {M}\"{o}bius function randomness and dynamics
  ({L}ecture 1).
\newblock
  {\footnotesize\texttt{http://publications.ias.edu/sites/default/files/MobiusFunctionsLectures(2).pdf}}.

\bibitem{Schmidt-Dyn-Sys-Alg-Or}
K.~Schmidt.
\newblock {\em Dynamical systems of algebraic origin}.
\newblock Modern Birkh\"auser Classics. Birkh\"auser/Springer Basel AG, Basel,
  1995.
\newblock [2011 reprint of the 1995 original] [MR1345152].

\bibitem{Tsang-1986}
K.M. Tsang.
\newblock The distribution of {$r$}-tuples of squarefree numbers.
\newblock {\em Mathematika}, 32(2):265--275 (1986), 1985.

\bibitem{vonNeumann-1932}
J.~von Neumann.
\newblock Zur {O}peratorenmethode in der klassischen {M}echanik.
\newblock {\em Ann. of Math. (2)}, 33(3):587--642, 1932.

\bibitem{Zimmer-1976}
R.~J. Zimmer.
\newblock Ergodic actions with generalized discrete spectrum.
\newblock {\em Illinois J. Math.}, 20(4):555--588, 1976.

\end{thebibliography}
\end{document}